\documentclass[11pt]{amsart}
\usepackage{latexsym,amssymb,amsmath,youngtab}
\usepackage{amsmath,amsthm,amssymb,amscd}
\usepackage[mathscr]{eucal}
\usepackage{verbatim}
\usepackage{epsfig}
\usepackage{color}
\usepackage{epsf,float}
\newtheoremstyle{custom}
  {3pt}
  {3pt}
  {\slshape}
  {}
  {\bfseries}
  {.}
  { }
   {}
\theoremstyle{custom}
\newtheorem{theorem}{Theorem}[subsection]

\newtheorem{proposition}[theorem]{Proposition}
\newtheorem{proposition/definition}[theorem]{Proposition/Definition}
\newtheorem{lemma}[theorem]{Lemma}
\newtheorem{corollary}[theorem]{Corollary}

\theoremstyle{definition}
\newtheorem{definition}[theorem]{Definition}

\newtheorem{example}[theorem]{Example}

\theoremstyle{remark}
\newtheorem{remark}[theorem]{Remark}

\newtheoremstyle{exercise}
  {3pt}
  {6pt}
  {}
  {}
  {\bfseries}
  {:}
  { }
   {}
\theoremstyle{exercise}
\newtheorem{exercise}[theorem]{Exercise}
\newtheoremstyle{exercises}
  {3pt}
  {6pt}
  {}
  {}
  {\bfseries}
  {:}
  {\newline}
   {}

\theoremstyle{exercise}
\newtheorem{exercises}[theorem]{Exercises}

\def\intprod{\negthinspace
\mathbin{\raisebox{.4ex}{\hbox{\vrule height .5pt width 4pt depth 0pt %
          \vrule height 4pt width .5pt depth 0pt}}}}

\input epsf
\def\boxit#1{\vbox{\hrule height1pt\hbox{\vrule width1pt\kern3pt
  \vbox{\kern3pt#1\kern3pt}\kern3pt\vrule width1pt}\hrule height1pt}}

\def\overarrow{\vec}
\def\trank{\text{rank}}

\def\BC{\mathbb C}

\def\BP{\mathbb P}\def\BG{\mathbb G}
\def\pp#1{\mathbb P^{#1}}

\def\pp#1{{\mathbb P}^{#1}}
\def\tdim{{\rm dim}}

\def\hd{,...,}
\def\ww{\wedge}
\def\upperp{{}^\perp}

\def\cS{{\mathcal S}}

\def\cO{{\mathcal O}}

\def\11{\mathbf 1}

\def\fsl{{\mathfrak {sl}}}

\def\fe{{\mathfrak e}}

\def\fg{{\mathfrak g}}

\def\fp{{\mathfrak p}}

\def\l{\lambda}
\def\a{\alpha}

\def\o{\omega}

\def\O{\Omega}
\def\b{\beta}
\def\g{\gamma}
\def\s{\sigma}

\def\d{\delta}

\def\up#1{{}^{({#1})}}

\def\ot{{\mathord{ \otimes } }}
\def\op{{\mathord{\,\oplus }\,}}

\def\ra{{\mathord{\;\rightarrow\;}}}

\def\dim{{\rm dim}\;}
\def\La#1{\Lambda^{#1}}

\def\overarrow{\vec}
\def\frak{\mathfrak}
\def\fsl{\frak s\frak l}

\def\op{\oplus}

\def\op{\oplus}

\def\s{\sigma}

\def\a{\alpha}
\def\b{\beta}

\def\g{\gamma}
\def\l{\lambda}

\def\ol{\overline}

\def\BP{\mathbb  P}
\def\BC{\mathbb  C}

\def\pp#1{\mathbb  P^{#1}}

\def\tcodim{\text{codim}}

\def\fp{\mathfrak  p}

\def\fg{\mathfrak  g}

\def\hd{, \hdots ,}

\def\La#1{\Lambda^{#1}}

\def\pp#1{\mathbb  P^{#1}}

\def\ra{\rightarrow}

\def\tdeg{\operatorname{deg}}
\def\tdet{\operatorname{det}}\def\tpfaff{\operatorname{Pf}}

\def\tim{\operatorname{Image}}
\def\tdim{\operatorname{dim}}
\def\tker{\operatorname{ker}}

\def\tmin{\operatorname{min}}

\def\trank{\operatorname{rank}}

\def\up#1{{}^{({#1})}}
\def\upperp{{}^{\perp}}

\def\ww{\wedge}

\def\be{\begin{equation}}
\def\ene{\end{equation}}

\def\tzeros{{\rm Zeros}}

\def\trank{{\rm rank}}

 \def\rig#1{\smash{ \mathop{\longrightarrow}
    \limits^{#1}}}

\def\dow#1{\Big\downarrow
   \rlap{$\vcenter{\hbox{$\scriptstyle#1$}}$}}
\def\up#1{\Big\uparrow
   \rlap{$\vcenter{\hbox{$\scriptstyle#1$}}$}}

\def\O{{\mathcal O}}
\def\P#1{\BP^#1}
\newcommand{\qedd}{\hfill\framebox[2mm]{\ }\medskip}

\def\oar{\overarrow}
\def\dtwo{\lfloor \frac d2\rfloor}

 \def\rig#1{\smash{ \mathop{\longrightarrow}
    \limits^{#1}}}

\def\dow#1{\Big\downarrow
   \rlap{$\vcenter{\hbox{$\scriptstyle#1$}}$}}
\def\up#1{\Big\uparrow
   \rlap{$\vcenter{\hbox{$\scriptstyle#1$}}$}}

\def\O{{\mathcal O}}
\def\P#1{\BP^#1}
 
\def\oar{\overarrow}
\def\dtwo{\lfloor \frac d2\rfloor}

\def\La#1{\wedge^{#1}}

\begin{document}

\title{Equations for secant varieties of Veronese and other varieties}
 \author{J.M. Landsberg and Giorgio Ottaviani}
\begin{abstract} New classes of modules of equations for secant varieties of
Veronese varieties are defined  using representation theory and geometry. Some
old modules of equations (catalecticant minors) are revisited to determine when they are sufficient to give scheme-theoretic
defining equations. An algorithm to decompose a general ternary quintic as the sum of seven fifth powers
is given as an illustration of our methods.
Our new equations and results about them are put
into a larger context by introducing    vector bundle techniques for finding equations of
secant varieties in general. We include  a few homogeneous examples of this method. 
\end{abstract}
 \thanks{First author  supported by NSF grants  DMS-0805782 and 1006353. Second author is member of GNSAGA-INDAM.}
\email{jml@math.tamu.edu, ottavian@math.unifi.it}
\maketitle

\section{Introduction}

\subsection{Statement of problem and main results}
Let $S^d\BC^{n+1}=S^dV$ denote the space of homogeneous polynomials of degree $d$ in $n+1$ variables, equivalently the
space of symmetric $d$-way tensors over $\BC^{n+1}$. It is an important problem for complexity theory,
signal processing, algebraic statistics, and many other areas (see  e.g., 
\cite{BCS,MR2447451,MR2205865,MR2383305}) to find tests for
the {\it border rank} of a given tensor. Geometrically, in the symmetric case,  this amounts to finding set-theoretic defining
equations for the {\it secant varieties} of the {\it Veronese variety} $v_d(\BP V)\subset \BP V$, the
variety of rank one symmetric tensors.
  
For an algebraic variety $X\subset \BP W$, the $r$-th secant variety $\s_r(X)$ is defined
by
\begin{equation}
\s_r(X)     = \overline{ \bigcup_{x_1\hd x_r\in X}\BP \langle x_1\hd x_r\rangle }\subset \BP W
\end{equation}
where $\langle x_1\hd x_r\rangle\subset W$ denotes the linear span of the points $x_1\hd x_r$ and the overline
denotes Zariski closure.
When $X=v_d(\BP V)$, $\s_r(X)$ is the Zariski closure of the set of polynomials that
are the sum of $r$ $d$-th powers.

When $d=2$,   $S^2V$ may  be thought of as the space of $(n+1)\times (n+1)$ symmetric matrices
 via the inclusion $S^2V\subset V\otimes V$ and the equations for $\s_r(v_2(\BP V))$ are just the size $r+1$ minors (these equations
even generate the ideal). The first equations found  for secant varieties of higher Veronese varieties
were obtained by imitating this construction, 
  considering the inclusions $S^dV\subset S^aV\otimes S^{d-a}V$, where $1\leq a\leq \lfloor \frac d2\rfloor$: Given $\phi\in S^dV$,
one considers the corresponding linear map $\phi_{a,d-a}: S^aV^*\ra S^{d-a}V$ and 
  if $\phi\in \s_r(v_d(\BP V))$, then  $\trank (\phi_{a,d-a})\leq
r$, see \S\ref{flatsubsect}. Such equations are called {\it minors of symmetric flattenings} or {\it catalecticant minors},
and date back at least to  Sylvester  who coined the term \lq\lq catalecticant\rq\rq .
See \cite{catalecticant} for a history.

These equations are usually both too numerous and too few, that is,  there are redundancies
among them and even all of them usually will not give enough equations to
define $\s_r(v_d(\BP V))$ set-theoretically.

In this paper we
\begin{itemize}

\item Describe a large class of new sets of equations for $\s_r(v_d(\BP V))$, which we call
{\it Young Flattenings}, that generalize  the classical {\it Aronhold invariant}, see Proposition \ref{mosect}.

\item Show a certain  Young Flattening, $YF_{d,n}$, provides scheme-theoretic equations
for a large class of cases where usual flattenings fail,  see Theorem \ref{mainyoung}.  

\item 
Determine cases where  flattenings are sufficient to give   defining equations (more precisely, scheme-theoretic equations), 
Theorem \ref{bigflatthm}. Theorem \ref{bigflatthm}
is primarily a consequence of  work  of A. Iarrobino and V. Kanev \cite{MR1735271} and Diesel \cite{diesel}.

\item Put our results in a larger context by providing a uniform formulation of all known equations for secant varieties via vector
bundle methods. We use this  perspective to prove some of our results, including a key induction
Lemma \ref{zeta}. The  discussion of vector bundle methods is postponed to the latter part of the
paper to make the results on symmetric border rank more accessible to readers outside of algebraic geometry.

\end{itemize}

Here is a chart summarizing what is known about  equations of secant varieties
of Veronese varieties:

{\small
$$
\begin{array}{|l|c|c|c|}
\hline
{\bf case} & {\bf   equations} & {\bf cuts\  out}&{\bf reference} \\
\hline
\s_r(v_2(\pp n))& {\rm size}\ r+1 {\rm\ minors} &   ideal & \rm{classical}\\
\hline
\s_r(v_d(\pp 1))&  {\rm size}\ r+1    {\rm\ minors \ of\ any\ } \phi_{s,d-s} &   ideal& \textrm{ Gundelfinger,
\cite{MR1735271}} \\
\hline
\s_2(v_d(\pp n))&  \begin{matrix} {\rm size}\ 3  {\rm\ minors \ of\ any\ }\\ \phi_{1,d-1}{\rm\ and\ }\phi_{2,d-2}
\end{matrix} &   ideal&
\textrm{ \cite{kanev}} \\
\hline
\s_3(v_3(\pp n))& {\rm Aronhold\ } +{\rm \ size\ } 4  {\rm\ minors \ of\ } \phi_{1,2} &   ideal&\textrm{Prop. \ref{syminherit}} \\
&&&\textrm{Aronhold for\ }n=2\textrm{\cite{MR1735271}}\\
\hline
\s_3(v_d(\pp n)), d\ge 4&   {\rm size\ } 4 {\rm\ minors \ of }\ \phi_{2,2} {\rm\ and\ } \phi_{1,3} &   scheme&
\textrm{Thm.\ref{bigflatthm} (\ref{s3vdpnthm}}) \\
&&&\textrm{\cite{schreyer}}\ \textrm{  for\ }n=2,d=4\\
\hline
\s_4(v_d(\pp 2))&   {\rm size\ } 5  {\rm\ minors \ of\ } \phi_{a,d-a}, a=\dtwo &   scheme& \textrm{Thm. \ref{bigflatthm} 
(\ref{s45vdp2thm})} \\
&&&\textrm{\cite{schreyer}}\ \textrm{ for\ }d=4\\
\hline
\s_5(v_d(\pp 2)),d\ge 6{\rm\ and\ }d=4& {\rm size\ } 6  {\rm\ minors \ of\ } \phi_{a,d-a}, a=\dtwo  &   scheme& \textrm{Thm. 
\ref{bigflatthm} (\ref{s5vdp2thm})}\\
&&&\textrm{Clebsch for\ }d=4\textrm{\cite{MR1735271}}\\
\hline
\s_r(v_5(\pp 2)), r\le 5&  {\rm size\ } 2r+2   {\rm\ subPfaffians \ of }\phi_{31,31} &   irred. comp.& \textrm{Thm. \ref{s6v5p2}}\\
\hline
\s_6(v_5(\pp 2))&  {\rm size\ } 14   {\rm\ subPfaffians \ of }\phi_{31,31} &   scheme& \textrm{Thm. \ref{s6v5p2}}\\
\hline
\s_6(v_d(\pp 2)), d\ge 6&  {\rm size\ }7  {\rm\ minors \ of } \phi_{a,d-a}, a=\dtwo   &   scheme& \textrm{Thm. 
\ref{bigflatthm} (\ref{s6vdp2thm}})\\
\hline
\s_7(v_6(\pp 2))& \textrm{symm.flat. + Young flat.} &   irred.comp.
& \textrm{Thm. \ref{sextics}}\\
\hline
\s_8(v_6(\pp 2))&  \textrm{symm.flat. + Young flat.} &irred.comp.& \textrm{Thm. \ref{sextics}}\\
\hline
\s_9(v_6(\pp 2))&   \det\phi_{3,3} &   ideal&
\textrm{classical} \\
\hline
\s_j(v_7(\pp 2)), j\le 10& {\rm\ size\ }2j+2{\rm \ subPfaffians \ of }\phi_{41,41}  &   irred.comp.& \textrm{Thm. \ref{mainyoung}}\\
\hline
\s_j(v_{2\d}(\pp 2)),j\le{{\d+1}\choose 2}&\begin{matrix} {\rm rank }\phi_{a,d-a}=\min(j,{{a+2}\choose 2}),\\
 1\le a\le \d \end{matrix} & 
  scheme & \textrm{\cite{MR1735271}, Thm. 4.1A}\\
&\rm{open\ and\ closed\ conditions}&&\\
\hline
\s_j(v_{2\d+1}(\pp 2)),j\le{{\d+1}\choose 2}+1&\begin{matrix} {\rm rank}\phi_{a,d-a}=\min(j,{{a+2}\choose 2}),
\\ 1\le a\le \d \end{matrix} & 
  scheme &\textrm{ \cite{MR1735271}, Thm. 4.5A}\\
&\rm{open\ and\ closed\ conditions}&&\\
\hline
\s_j(v_{2\d}(\pp n)),j\le{{\d+n-1}\choose n}&{\rm size}\ j+1 {\rm\ minors\ of} \phi_{\d,\d}&   irred. comp. & \textrm{\cite{MR1735271}  Thm. 4.10A}\\
\hline
\s_j(v_{2\d+1}(\pp n)),j\le{{\d+n}\choose n}&\begin{matrix} {\rm size}\ {n\choose a}j+1 {\rm\ minors\ of\ } Y_{d,n},
\\  a=\lfloor n/2\rfloor\end{matrix} &   irred. comp. &
 \textrm{Thm. \ref{mainyoung}}\\
&\begin{matrix} \textrm{if\ }n=2a, a\textrm{\ odd}, {n\choose a}j+2\\
 {\rm\ subpfaff.\ of\ } Y_{d,n}\end{matrix} &&\\
\hline
\end{array}
$$
}

\subsection{Young Flattenings}\label{YFsect}
The simplest case of equations for secant varieties is for the space of rank at most $r$ matrices of size $p\times q$, which
is the zero set of the minors of size $r+1$. Geometrically let $A=\BC^p$, $B=\BC^q$ and let
$Seg(\BP A\times \BP B)\subset \BP (A\ot B)$ denote the Segre variety of rank one matrices. Then the ideal 
of $\s_r(Seg(\BP A\times \BP B))$ is generated by the space of minors of size $r+1$, which is $\La{r+1}A^*\ot \La{r+1}B^*$.
Now if $X\subset \BP W$ is a variety, and there is a linear injection  $W\ra A\ot B$ such that $X\subset \s_p(Seg(\BP A\times \BP B))$,
then the minors of size $pr+1$ furnish equations for $\s_r(X)$.
Flattenings are a special case of this method where $W=S^dV$, $A=S^aV$ and $B=S^{d-a}V$.

When a group $G$ acts linearly on $W$ and $X$ is invariant under the group action, then the equations of
$X$ and its secant varieties will be $G$-modules, and one looks for a $G$-module map $W\ra A\ot B$.
 Thus one looks for $G$-modules $A,B$ such that $W$ appears in the $G$-module decomposition of $A\ot B$.
We discuss this in detail for $X=v_d(\BP V)$ in \S\ref{yflatsect} and in general in \S\ref{grassexsect} and \S\ref{gpexamsect}. For now we focus on a special class of Young
flattenings that we describe in elementary language. We begin by reviewing the classical {\it Aronhold invariant}.

\begin{example}\label{aronex}[The Aronhold invariant] The classical
Aronhold invariant  is the equation for the hypersurface
$\s_3(v_3(\pp 2))\subset \pp 9$.   Map $S^3V\ra (V\ot \La 2 V)\ot (V\ot V^*)$, by
first embedding $S^3V\subset V\ot V\ot V$, then tensoring with $Id_V\in V\ot V^*$, and then skew-symmetrizing.
Thus, when  $n=2$,  $\phi\in S^3V$ gives rise to an element of $\BC^9\ot \BC^9$. In bases, if we write

\begin{align*}
\phi= &\phi_{000}x_0^3+\phi_{111}x_1^3+\phi_{222}x_2^3+3\phi_{001}x_0^2x_1+3\phi_{011}x_0x_1^2+3\phi_{002}x_0^2x_2\\
&+3\phi_{022}x_0x_2^2
+3\phi_{112}x_1^2x_2+3\phi_{122}x_1x_2^2+6\phi_{012}x_0x_1x_2,
\end{align*} 
the corresponding matrix is:

$$\left[\begin{array}{ccccccccc}
&&&\phi_{002}&\phi_{012}&\phi_{022}&-\phi_{010}&-\phi_{011}&-\phi_{012}\\
&&&\phi_{012}&\phi_{112}&\phi_{122}&-\phi_{011}&-\phi_{111}&-\phi_{112}\\
&&&\phi_{012}&\phi_{112}&\phi_{222}&-\phi_{012}&-\phi_{112}&-\phi_{122}\\
-\phi_{002}&-\phi_{012}&-\phi_{022}&&&&\phi_{000}&\phi_{001}&\phi_{002}\\
-\phi_{012}&-\phi_{112}&-\phi_{122}&&&&\phi_{001}&\phi_{011}&\phi_{012}\\
-\phi_{012}&-\phi_{112}&-\phi_{222}&&&&\phi_{001}&\phi_{011}&\phi_{022}\\
\phi_{010}&\phi_{011}&\phi_{012}&-\phi_{000}&-\phi_{001}&-\phi_{002}&&&\\
\phi_{011}&\phi_{111}&\phi_{112}&-\phi_{001}&-\phi_{011}&-\phi_{012}&&&\\
\phi_{012}&\phi_{112}&\phi_{122}&-\phi_{001}&-\phi_{011}&-\phi_{022}&&&\\
\end{array}\right].
$$

All the principal Pfaffians of size $8$ of the this matrix coincide, up to scale,
with the classical  {\it Aronhold invariant}. (Redundancy occurs here is because one
should really work with the submodule $S_{21}V\subset V\ot \La 2 V\simeq V\ot V^*$, where
the second identification uses a choice of volume form.  The Pfaffian of the map
$S_{21}\to S_{21}$  is the desired equation.)

 This construction, slightly different from the one in \cite{ottwaring},  shows  how the Aronhold invariant
is analogous to the invariant in    $S^9(\BC^3\ot \BC^3\ot\BC^3)$
that was  discovered   by Strassen \cite{Strassen505}, (see also \cite{ottrento},  and the paper \cite{MR0460330} 
by Barth, all in different settings.)
\end{example}

\medskip

Now consider the 
  inclusion $V\subset \La k V^*\ot \La{k+1}V$, given by  $v\in V$ maps to
the map $\o\mapsto v\ww \o$. In bases one obtains a matrix whose
entries are the coefficients of  $v$  or zero. In the special case $n+1=2a+1$ is odd
and $k=a$, one obtains a square matrix $K_n$, which is skew-symmetric for odd $a$ and
symmetric for even $a$.  For example, when
$n=2$, the matrix is
$$K_2=\begin{pmatrix}0&
      {x}_{2}&
      {-{x}_{1}}\\
      {-{x}_{2}}&
      0&
      {x}_{0}\\
      {x}_{1}&
      {-{x}_{0}}&
      0\\
      \end{pmatrix}
$$
 and, when $n=4$, the matrix is

$$K_4=\left[\begin{array}{rrrrrrrrrr}
&&&&&&&x_4&-x_3&x_2\\
&&&&&-x_4&x_3&&&-x_1\\
&&&&x_4&&-x_2&&x_1\\
&&&&-x_3&x_2&&-x_1\\
&&x_4&-x_3&&&&&&x_0\\
&-x_4&&x_2&&&&&-x_0\\
&x_3&-x_2&&&&&x_0\\
x_4&&&-x_1&&&x_0\\
-x_3&&x_1&&&-x_0\\
x_2&-x_1&&&x_0\\
\end{array}\right].$$

Finally,  consider the following generalization of both the Aronhold invariant and the $K_n$. 
Let  $a=\lfloor \frac{n}{2}\rfloor$ and  let $d=2\d+1$.
Map $S^dV\ra (S^{\d}V\ot \La aV^*) \ot (S^{\d }V\ot\La{a+1}V) $ by
first performing the inclusion $S^dV\ra S^{\d}V\ot S^{\d}V\ot V$
and then using the last factor to obtain a map 
$\La a V\ra \La{a+1}V$.
 We get:
\be\label{YFmap}
YF_{d,n}(\phi) : S^{\d}V^*\ot \La a V\ra S^{\d}V\ot \La{a+1}V.
\ene
If $n+1$ is odd, the matrix representing $YF_{d,n}(\phi)$
 is skew-symmetic, so we may take Pfaffians instead of minors.

For a decomposable $w^d\in S^dV$, the map is
$$
\a^{\d}\ot v_1\ww\cdots \ww v_a \mapsto
(\a(w))^{\d} w^{\d} \ot w\ww v_1\ww\cdots \ww v_a.
$$
In bases, one obtains a matrix in block form, where the
blocks correspond to  the entries of $K_n$ and the matrices in the blocks
are the square catalecticants $\pm (\frac{\partial\phi}{\partial x_i})_{\d,\d}$ in the place of $\pm x_i$.  

Let
$$
YF_{d,n}^r:=\{\phi\in S^dV\mid \trank(YF_{d,n}(\phi))\leq \binom n{\lfloor\frac n2\rfloor}r\}.
$$
Two interesting cases are $YF_{3,2}^3=\s_3(v_3(\pp 2))$ which defines the quartic Aronhold invariant
and $YF_{3,4}^7=\s_7(v_3(\pp 4))$ which defines the invariant of degree $15$ considered in \cite{ottwaring}.

\begin{remark} Just as with the Aronhold invariant above, there will be redundancies
among the minors and Pfaffians of $YF_{d,n}(\phi)$. See \S\ref{yflatsect} for a description without
redundancies.
\end{remark}

\begin{theorem}\label{mainyoung} 

Let $n\geq 2$, let $a=\lfloor \frac{n}{2}\rfloor$,  let $V=\BC^{n+1}$, and let $d=2\d+1$.

If $r\le {{\d+n}\choose{n}}$ then
$\s_r(v_d(\pp n))$ is an  irreducible component
of $YF_{d,n}^r$,    the variety given by the size ${{n}\choose a}r+1$ minors of $YF_{d,n}$.

In the case $n=2a$ with odd $a$,  $YF_{d,n}$ is skew-symmetric (for any $d$)
and one may instead  take the size ${{n}\choose a}r+2$ sub-pfaffians of $YF_{d,n}$.

In the case $n=2a$ with even $a$,   $YF_{d,n}$ is symmetric.
 \end{theorem}

\medskip

The bounds given in   Theorem \ref{mainyoung} for $n=2$ are sharp (see Proposition \ref{boundodd}).

\subsection{Vector bundle methods}

As mentioned above, the main method  for finding equations of secant varieties for
$X\subset \BP V$ is to find a linear embedding
$V\subset A\ot B$, where $A,B$ are vector spaces,
such that $X\subset \s_q(Seg(\BP A\times \BP B))$, 
where $Seg(\BP A\times \BP B)$ denotes the Segre
variety of rank one elements. What follows is a
technique  to find such inclusions using vector bundles.

 Let $E$ be a vector bundle on $X$ of rank $e$, write $L=\cO_X(1)$, so $V=H^0(X,L)^*$.  Let $v\in V$ and  consider the linear map
\begin{align}\label{pairingmap}
A_{v}^E\colon H^0(E)&\to H^0(E^{*}\otimes L)^{*}
\end{align}
{induced by the natural map
$$A\colon H^0(E)\otimes H^0(L)^*\to H^0(E^*\otimes L)^*$$}
where $A^E_v(s)=A(s\ot v)$. 
In   examples $E$ will be chosen so that both
$H^0(E)$ and $H^0(E^{*}\otimes L)$ are nonzero, otherwise our construction is vacuous.
A key observation (Proposition \ref{construct}) is that 
{\it the size $(r e+1)$  minors of $A_v^E$   give equations for $\sigma_r(X)$}.

\subsection{Overview}
In \S\ref{backgroundsect} we establish notation and collect standard facts that we will need later.  In \S\ref{catsect}, we
first review work of Iarrobino and Kanev \cite{MR1735271} and S. Diesel \cite{diesel}, then show how their results
imply several new cases where $\s_r(v_d(\BP V))$ is cut out scheme-theoretically by flattenings (Theorem
\ref{bigflatthm}). 
 In \S\ref{yflatsect} we discuss Young Flattenings for Veronese varieties. The possible
inclusions $S^dV\ra S_{\pi}V\ot S_{\mu}V$ follow easily from the Pieri formula, however which of these
are useful is still not understood.     In \S\ref{surfacecasesect}
we make a detailed study of the $n=2$ case. 
The above-mentioned  Proposition \ref{construct} is proved in \S\ref{constsect}, where we also describe simplifications
when $(E,L)$ is a symmetric or skew-symmetric pair. We also give 
a sufficient criterion for $\s_r(X)$ to
be an irreducible component of the equations given by the 
$(r e+1)$  minors of $A_v^E$ (Theorem \ref{irredcrit}). In \S\ref{indthmsect} we prove a downward
induction lemma (Lemma \ref{zeta}). We prove Theorem \ref{mainyoung} in \S\ref{mainypfsect}, which includes
Corollary \ref{hypsingular} on linear systems of hypersurfaces, which may be of interest in its own right. To get explicit models for the maps $A^E_v$ it is sometimes
useful to factor $E$, as described in \S\ref{factoringsect}.
  In \S\ref{decompsect} we explain how to use equations to obtain decompositions
of polynomials into sums of powers, illustrating with    an algorithm to decompose a general ternary quintic as the sum of seven fifth powers. 
  We conclude, in \S\ref{grassexsect}-\ref{gpexamsect} with a few brief
examples of the construction for homogeneous varieties beyond Veronese varieties.

\subsection*{Acknowledgments}
We thank P. Aluffi, who pointed out the refined B\'ezout theorem \ref{refinedbezout}.
This paper grew out of questions raised at the
2008 AIM workshop {\it Geometry and representation theory of tensors
for computer science, statistics and other areas}, and the authors
thank AIM and the conference participants for inspiration.

\section{Background}\label{backgroundsect}

\subsection{Notation}
We work exclusively over the complex numbers.
$V,W$ will generally denote (finite dimensional) complex vector spaces.
{ The dual space of $V$ is denoted $V^*$. The projective space of lines through the origin of $V$ is denoted by $\BP V$ .} If $A\subset W$ is a subspace   $A^{\perp}\subset W^*$ is its annihilator,
the space  of   $f\in W^*$ such that $f(a)=0$ $\forall a\in A$.

For  a partition $\pi=(p_1\hd p_r)$ of $d$,   we
write $|\pi|=d$ and $\ell(\pi)=r$. 
If  $V$ is a vector space,  $S_{\pi}V$ denotes the  
irreducible 
$GL(V)$-module  determined by $\pi$  (assuming $\tdim V\geq\ell(\pi)$).
In particular
$S^dW=S_{(d)}W$ and $\wedge^aW=S_{1^a}W$ are respectively the $d$-th symmetric power and the $a$-th exterior power of $W$.
$S^dW=S_{(d)}W$ is also the space of homogeneous polynomials
of degree $d$ on $W^*$. Given $\phi\in S^dW$, $\tzeros(\phi)\subset \BP W^*$ denotes its zero set. 


For a subset $Z\subseteq \BP W$,   $\hat Z\subseteq W\backslash 0$ denotes the affine cone over $Z$.
For a projective variety  $X\subset \BP W$, $I(X)\subset Sym(W^*)$ denotes its ideal and
$I_X$ its ideal sheaf of   (regular) functions vanishing at $X$. 
For a smooth point $z\in X$, $\hat T_zX\subset V$ is the affine tangent space and
$\hat N^*_zX=(\hat T_zX)\upperp\subset V^*$ the affine conormal space.

We make the standard identification of a vector bundle  with the corresponding locally
free sheaf. 
For a sheaf $E$ on $X$,   $H^i(E)$  is the $i$-th cohomology space of $E$. In particular
$H^0(E)$ is the space of global sections of $E$, and $H^0(I_Z\otimes E)$
is the space of global sections of $E$ which vanish on a subset  $Z\subset X$.
{According to this notation $H^0(\BP V,\O(1))=V^*$.}

If $G/P$ is a rational homogeneous variety and  $E\ra G/P$ is an irreducible
homogeneous  vector bundle, we write $E=E_{\mu}$ where $\mu$ is the
highest weight of the irreducible $P$-module inducing $E$. We use the conventions
of \cite{MR1890629} regarding roots and  weights of simple Lie algebras. 

\subsection{Flattenings}\label{flatsubsect}Given $\phi\in S^dV$, write $\phi_{a,d-a}\in S^aV\ot S^{d-a}V$
for the $(a,d-a)$-polarization of $\phi$. We often consider
$\phi_{a,d-a}$ as a linear map $\phi_{a,d-a}: S^aV^*\ra S^{d-a}V$.
This notation is compatible with the more general one of Young flattenings that we will introduce in \S\ref{yflatsect}.

If $[\phi]\in v_d(\BP V)$ then for all $1\leq a\leq d-1$,  $\trank(\phi_{a,d-a})=1$, and    the $2\times 2$ minors
of $\phi_{a,d-a}$ generate the ideal of $v_d(\BP V)$ for any $1\leq a\leq d-1$. If $[\phi]\in\s_r(v_d(\BP V))$, then   $\trank(\phi_{a,d-a})\leq r$, so, the $(r+1)\times
(r+1)$ minors of $\phi_{a,d-a}$ furnish equations
for $\s_r(v_d(\BP V))$, i.e., 
$$
\La{r+1}(S^aV^*)\ot \La{r+1}(S^{d-a}V^*)\subset I_{r+1}(\s_r(v_d(\BP V))).
$$
Since $ I_{r }(\s_r(v_d(\BP V)))=0$, these modules,  obtained by {\it symmetric flattenings}, also called   {\it catalecticant homomorphisms},
 are
among the modules generating the ideal of $\s_r(v_d(\BP V))$. Geometrically the symmetric
flattenings are the equations for the varieties
$$
Rank^r_{a,d-a}(S^dV):=\s_r(Seg(\BP S^aV\times \BP S^{d-a}V))\cap \BP S^dV.
$$

\begin{remark}\label{splitflatrem}
The equations of $\s_r(v_2(\BP W))\subset \BP S^2W$ are those of
$\s_r(Seg(\BP W\times \BP W))$ restricted to $\BP S^2W$.
Since $S^{2p}V \subset S^2(S^pV)$,  when
$d=2a$ we may also describe the symmetric flattenings as the equations for
$\s_r(v_2(\BP S^aV))\cap \BP S^dV 
$.
\end{remark}

\subsection{Inheritance} The purpose of this subsection  is to explain why it is only necessary  to consider the \lq\lq primitive\rq\rq\ cases
of $\s_r(v_d(\pp n))$ for  {$n\leq r-1$}.

Let 
\begin{align*}
Sub_r(S^dV):&=\BP \{\phi\in S^dV\mid \exists V'\subset V,\
\tdim V'=r,\ \phi\in S^dV'\}\\
&=  \{[\phi]\in \BP S^dV\mid \tzeros(\phi)\subset \BP V^*{\rm\ is \ a\ cone\ over\ a\ linear\ space \ of \ codimension\ }r\}
\end{align*}
denote the {\it subspace variety}.   
The ideal of $Sub_r(S^dV)$ is generated in 
degree $r+1$ by all modules $S_{\pi}V^*\subset S^{r+1}(S^dV^*)$
where $\ell(\pi)>r+1$, see \cite[\S 7.2]{weyman}.
These modules may be realized explicitly as the 
$(r+1)\times (r+1)$ minors of $\phi_{1,d-1}$.
Note in particular that  $\s_r(v_d(\BP V))\subset Sub_r(S^dV)$  and that
  equality holds for $d\le 2$ or $r=1$. 
 Hence the equations of $Sub_r(S^dV)$ appear among the equations of $\s_r(v_d(\BP V))$.
 
Let $X\subset \BP W$ be a $G$-variety for some group $G\subset GL(W)$. We recall that a module $M\subset Sym(W^*)$
defines $X$ {\it set-theoretically}   if $\tzeros(M)=X$ as a set, that it defines $X$
{\it scheme-theoretically }  if there exists a $\d$ such that
the ideal generated by $M$ equals the ideal of $X$ in all degrees greater than $\d$, and
that $M$ defines  $X$ {\it ideal theoretically} if the ideal generated by $M$ equals the ideal of $X$.

\begin{proposition} [Symmetric Inheritance]\label{syminherit} Let $V$ be a vector space of dimension greater than $r$.

Let
$M\subset Sym((\BC^r)^*)$ be a module and let $U=[\La{r+1}V^*\ot \La{r+1}(S^{d-1}V^*)]\cap S^{r+1}(S^dV^*) 
\subset S^{r+1}(S^dV^*)$ be the module generating the ideal of $Sub_r(S^dV)$ given by  the $(r+1)\times (r+1)$-minors
of the flattening $\phi\mapsto \phi_{1,d-1}$.

If $M$ defines $\s_r(v_d(\pp{r-1}))$ set-theoretically, respectively scheme-theoretically, resp. ideal-theoretically,
let $\tilde M\subset Sym(V^*)$ be the module induced by $M$. Then  $\tilde M+U$ defines
$\s_r(v_d(\BP V))$ set-theoretically, resp. scheme-theoretically, resp. ideal-theoretically.
\end{proposition}

See \cite[Chapter 8]{Ltensorbook}  for a proof.

\subsection{Results related to degree} Sometimes it is possible to conclude global information
from local equations if one has information about degrees.

We need   the following result about excess intersection.

\begin{theorem}\label{refinedbezout}
Let   $Z\subset\pp n$ be a variety of codimension $e$ and  $L\subset \pp n $ a linear subspace
of codimension $f$. Assume that $Z\cap L$ has an irreducible component  $Y$ of codimension $f+e\le n$
such that $\deg Y=\deg Z$. Then $Z\cap L=Y$.
\end{theorem}
\begin{proof} This is an application of the refined B\'ezout theorem
of  \cite[Thm. 12.3]{FultonIT}.
\end{proof}

The degrees of $\sigma_r(v_d(\BP^2))$ for certain small values of $r$ and $d$     were computed by Ellingsrud and Stromme
in \cite{MR1317230}, see also \cite[Rem. 7.20]{MR1735271}.

\begin{proposition}{\cite{MR1317230}}\label{112}
\begin{enumerate}
\item   $\tdeg(\sigma_3(v_4(\BP^2))) =112$.

\item $\tdeg(\s_4(v_4(\pp 2))=35$.

\item   $\tdeg(\sigma_6(v_5(\BP^2))) =140$.

\item $\tdeg(\s_6(v_6(\pp 2)))= 28,314$.
\end{enumerate}
\end{proposition}

\begin{proposition}(see, e.g., \cite[Cor. 3.2]{LMsec}) The minimal possible degree of a module
in $I(\s_r(v_d(\pp n)))$ is $r+1$.
\end{proposition}

We recall the following classical formulas, due to C. Segre. For a modern reference see
\cite{HaTu}.

\begin{proposition}\label{segre}
$$\textrm{codim\ }\sigma_r(v_2(\pp n))={{n-r+2}\choose{2}}\qquad\deg \sigma_r(v_2(\pp n))=\prod_{i=0}^{n-r}\frac{{{n+1+i}\choose{n-r-i+1}}}{{{2i+1}\choose{i}}}$$
 
$$\textrm{codim\ }\sigma_{r}(G (2,n+1))={{n-2r+1}\choose{2}}\quad\deg \sigma_{r}(G (2,n+1))=
\frac{1}{2^{n-2r}}\prod_{i=0}^{n-2r-1}\frac{{{n+1+i}\choose{n-2r-i}}}{{{2i+1}\choose{i}}}.
$$

 \end{proposition}
 
\subsection{Conormal spaces}\label{conormal} The method used to  prove a module of equations locally
defines $\s_r(v_d(\pp n))$ will be to show that the conormal space at a smooth point
of the zero set of the module equals the conormal space to $\s_r(v_d(\pp n))$ at that
point.

Let $A,B$ be vector spaces and let $Seg(\BP A\times \BP B)\subset \BP (A\ot B)$ denote the Segre variety, so
if $[x]\in Seg(\BP A\times \BP B)$, then $x= a\ot b $ for some $a\in A$ and $b\in B$.
One has the affine tangent space
$\hat T_{[x]} Seg(\BP A\times \BP B)=a\ot B+ A\ot b\subset A\ot B$, and the affine conormal
space $\hat N^*_{[x]} Seg(\BP A\times \BP B)=a\upperp\ot b\upperp= \tker(x)\ot \tim(x)\upperp \subset A^*\ot B^*$,
where in the latter description we think of $a\ot b: A^*\ra B$ as a linear map.
Terracini's lemma implies that if $[z]\in \s_r(Seg(\BP A\times \BP B))$ is of rank $r$,
then 
\be\label{conormaleqn}
\hat N^*_{[z]}\s_r( Seg(\BP A\times \BP B))= \tker(z)\ot \tim(z)\upperp.
\ene

In particular, letting $Rank^r_{(a,d-a)}(S^dV)\subset \BP (S^dV)$ denote the zero set of the
size $(r+1)$-minors of the flattenings $S^aV^*\ra S^{d-a}V$, one has

\begin{proposition}\label{conormalflat} Let $[\phi]\in Rank^r_{(a,d-a)}(S^dV)$ be a sufficiently general point, then
$$
\hat N^*_{[\phi]}Rank^r_{(a,d-a)}(S^dV)=\tker(\phi_{a,d-a})\circ \tim(\phi_{a,d-a})\upperp \subset S^dV^*.
$$
\end{proposition}

Now let $[y^d]\in v_d(\BP V)$, then 
\begin{align*}\hat N^*_{[y^d]}v_d(\BP V)&= \{ P\in S^dV^*\mid
 P(y)=0, \ dP_y=0\}\\
&=S^{d-2}V^*\circ S^2y\upperp\\
&=\{ P\in S^dV^*\mid \tzeros(P) {\rm{\ is\ singular\ at\ }} [y]\}.
\end{align*}

Applying Terracini's lemma yields:

\begin{proposition}\label{nstarsrvd}
Let $[\phi]=[y_1^d+\cdots + y_r^d]\in \s_r(v_d(\BP V))$.
Then
\begin{align*}
\hat N^*_{[\phi]} \s_r(v_d(\BP V))&\subseteq  (S^{d-2}V^*\circ S^2y_1\upperp) \cap \cdots\cap
(S^{d-2}V^*\circ S^2y_r\upperp)\\
&= \{ P\in S^dV^*\mid \tzeros(P) {\rm{\ is\ singular\ at\ }} [y_1]\hd [y_r]\} 
\end{align*}
and equality holds if $\phi$ is sufficiently general.
\end{proposition}

\section{Symmetric Flattenings (catalecticant minors)}\label{catsect}

\subsection{Review of known results}

The following result dates back to  the work of Sylvester and Gundelfinger.
 
\begin{theorem} (see, e.g., \cite[Thm. 1.56]{MR1735271})\label{gsthm}
The ideal of $\s_r(v_d(\pp 1))$ is generated in degree $r+1$
by the size $r+1$   minors of 
$\phi_{u,d-u}$ for any $r\leq u\leq d-r$, i.e.,
by any of the modules $\La{r+1}S^u\BC^2\ot \La{r+1}S^{d-u}\BC^2$.
The variety $\s_r(v_d(\pp 1))$ is projectively
normal and  arithmetically Cohen-Macaulay,  its singular
locus is $\s_{r-1}(v_d(\pp 1))$, and its degree is
$\binom{d-r+1}{r}$ .
\end{theorem}
 
\begin{corollary}   [Kanev, \cite{kanev}]\label{kanevcor}  The ideal of $\s_2(v_d(\pp n))$ is generated in degree $3$
by the $3$ by $3$ minors of the $(1,d-1)$ and $(2,d-2)$ flattenings.
\end{corollary}
\begin{proof}
Apply Proposition \ref{syminherit} to Theorem \ref{gsthm} in the case
$r=2$.
\end{proof}

{
\begin{remark} C. Raicu \cite{raicu} recently proved    that in Corollary \ref{kanevcor}
it is possible to replace the $(2,d-2)$ flattening with any $(i,d-i)$-flattening such that $2\le i\le d-2$.
\end{remark}
}



\begin{theorem}(see \cite[Thms. 4.5A, 4.10A]{MR1735271})
\label{maincata}
Let $n\geq 3$, and let $V=\BC^{n+1}$.   Let $\d=\llcorner \frac d2\lrcorner$, $\d'=\ulcorner \frac d2\urcorner$.
If $r\le {{\d-1+n}\choose{n}}$ then
$\s_r(v_d(\BP V))$ is an irreducible component of $\s_r(Seg(\BP S^{\d}V\times\BP S^{\d'}V))\cap \BP (S^{d}V)$.
In other words,  
$\s_r(v_d(\pp n))$ is an  irreducible component
of the  size $(r+1)$ minors of  the $(\d,\d')$-flattening. \end{theorem}

The bound obtained in this theorem is the best  we know of  for even degree, apart from  some cases of small degree   listed below.
Theorem \ref{mainyoung} is an improvement of this bound in the case of odd degree.

When  $\dim V=3$ then a Zariski open subset of  $\s_r(v_d(\BP V))$ can be characterized by open and closed conditions,
that is, with the same assumptions of Theorem \ref{maincata}:
If  the rank of all $(a,d-a)$-flattenings computed at $\phi$
is equal   to $\min\left(r,{{a+2}\choose 2}\right)$
then $\phi\in\s_r(v_d(\BP V))$ and the general element of $\s_r(v_d(\BP V))$
can be described in this way 
(see \cite[Thm. 4.1A]{MR1735271}).

\begin{remark} The statement of Theorem \ref{maincata}  cannot be improved,
in the sense that   that the zero locus of the flattenings usually is  reducible.
Consider the case $n=2$, $d=8$, $r=10$, in this case
the $10\times 10$ flattenings of $\phi_{4,4}$ define a variety with   at least two irreducible components,
one of them is $\s_{10}(v_8(\pp 2))$ (see \cite[Ex. 7.11]{MR1735271}).
\end{remark}


 \begin{proposition} Assume $\s_r(v_d(\pp n))$ is not
 defective and $r< \frac 1{n+1}\binom{n+d}d$ so that
$\s_r(v_d(\pp n))$ is not the ambient space. 
 Let $\d=\llcorner \frac d2\lrcorner$, $\d'=\ulcorner \frac d2\urcorner$.
 There are non-trivial equations from flattenings iff
 $
r< \binom{\d+n }{\d}$.
 (The defective cases
where $r$ is allowed to be larger are understood as well.)
 \end{proposition}

  \begin{proof}
  The right hand side is the size of the maximal minors of
 $\phi_{\d,\d'}$, which give equations for 
$ \s_{\binom{\d+n}{\d} -1}(v_d(\pp n))$.
\end{proof}

For example, when $n=3$  and  $d=2\d$, $\s_r(v_d(\pp 3))$  is
 an irreducible  component of the zero set of the flattenings 
 for $r\leq \binom{\d+2}3$, the flattenings give some equations
 up to $r\leq\binom{\d+3}3$, and there are non-trivial equations
 up to $r\leq \frac 14\binom{2\d+3}3$.

We will have need of more than one symmetric flattening, so we make the following
definitions, following \cite{MR1735271}, but modifying the notation:

\begin{definition}\label{sflatdef} Fix a sequence $\overarrow{r}:=(r_1\hd r_{\llcorner \frac d2\lrcorner})$
and let
\begin{align*}
SFlat_{\overarrow{r}}(S^dW):=&
\{ \phi\in S^dW \mid \trank (\phi_{j,d-j})\leq r_j, \ j=1\hd \llcorner \frac d2\lrcorner\}
\\
&=
[\bigcap_{j=1}^{\llcorner \frac d2\lrcorner}
\s_{r_j}(Seg(\BP S^jW \times \BP S^{d-j}W))] \cap S^dW
\end{align*}

Call a  sequence $\overarrow{r}$ {\it admissible} if there exists 
$\phi\in S^dW$ such that $\trank (\phi_{j,d-j})= r_j$ for all $j=1\hd \llcorner \frac d2\lrcorner$.
It is sufficient to consider  $\oar{r}$ that are admissible because if $\oar r$ is non-admissible,
the zero set of $SFlat_{\oar r}$ will be contained in the union of   admissible $SFlat$'s associated
to smaller $\oar r$'s in the natural partial order.
Note that $SFlat_{\overarrow{r}}(S^dW)\subseteq Sub_{r_1}(S^dW)$.
\end{definition}

Even if $\oar r$ is admissible, it still can be the case that
$SFlat_{\overarrow{r}}(S^dW)$ is reducible. 
For example, when $\tdim W=3$ and 
  $d\ge 6$,  the zero set of the size $5$ minors of the $(2,d-2)$-flattening
has  two irreducible components, one of them is $\sigma_4(v_d(\P 2))$
and the other  has dimension $d+6$ \cite[Ex. 3.6]{MR1735271}.

To remedy this,
let $\oar r$ be admissible, and consider
$$
SFlat_{\overarrow{r}}^0(S^dW):= 
\{ \phi\in S^dW \mid \trank (\phi_{j,d-j})= r_j, \ j=1\hd \llcorner \frac d2\lrcorner\}
$$
and let
$$
Gor(\oar r):=\ol{SFlat_{\overarrow{r}}^0(S^dW)}.
$$
 
\begin{remark} In   the commutative
algebra literature (e.g. \cite{diesel,MR1735271}),    \lq\lq Gor\rq\rq\  is short for Gorenstein, see \cite[Def. 1.11]{MR1735271} for a history.
\end{remark}

Unfortunately, defining equations for $Gor(\oar r)$ are not known.
One can test for membership of $SFlat_{\overarrow{r}}^0(S^dW)$
by checking the required vanishing and non-vanishing of minors.  

\begin{theorem}\label{oarrirred} \cite[Thm 1.1]{diesel}  If $\tdim W=3$, and $\oar r$ is admissible,
then $Gor(\oar r)$ is irreducible.
\end{theorem}


Theorem \ref{oarrirred} combined with Theorem \ref{maincata} allows one  to extend the set of secant varieties
of Veronese varieties defined by flattenings.

\subsection{Consequences of Theorems \ref{oarrirred} and \ref{maincata}}

\begin{theorem}\label{bigflatthm} The following varieties are defined scheme-theoretically by minors of flattenings:
\begin{enumerate}

\item \label{s3vdpnthm}
Let $d\ge 4$. The variety $\sigma_3(v_d(\P n))$ is defined scheme-theoretically
by the $4\times 4$ minors of the $(1,d-1)$ and $(\dtwo, d-\dtwo)$ flattenings.

\item \label{s45vdp2thm} 
   For $d\ge 4$ the variety $\sigma_4(v_d(\P 2))$ is defined scheme-theoretically
by the $5\times 5$ minors of the 
 $(\dtwo, d-\dtwo)$ flattenings.

\item \label{s5vdp2thm} For $d\ge 6$ the variety $\sigma_5(v_d(\P 2))$ is defined scheme-theoretically
by the $6\times 6$ minors of the  $(\dtwo, d-\dtwo)$ flattenings.

 \item \label{s6vdp2thm} Let $d\ge 6$. The variety $\sigma_6(v_d(\P 2))$ is defined scheme-theoretically
by the $7\times 7$ minors of the $(\dtwo,d-\dtwo)$ flattenings. 
\end{enumerate}
\end{theorem}

\begin{remark} In the recent preprint \cite{bucbuc} it is proved that the variety
 $\sigma_r(v_d(\P n))$ is defined set-theoretically
by the $(r+1)\times (r+1)$ minors of the  $(i, d-i)$ flattenings
for $n\le 3$, $2r\le d$ and $r\le i\le d-r$.
    Schreyer proved  \eqref{s45vdp2thm} in the case $d=4$, 
\cite[Thm. 2.3]{schreyer}. 
\end{remark}

By \cite[Thm 4.2]{MR0485835}, when $n\leq 2$, 
$\trank(\phi_{s,d-s})$ is nondecreasing in $s$ for $1\leq s\leq \lfloor \frac d2\rfloor$. 
 We will use this fact often in this section.

To prove each case, we simply  show the scheme defined by the flattenings
in the hypotheses coincides with the scheme of some $Gor(\oar r)$ with $\oar r$ admissible and
then Theorem \ref{maincata} combined with Theorem \ref{oarrirred} implies the result. 
The first step is to determine which $\oar r$ are admissible.

\begin{lemma}\label{admisseqlem}  
The only admissible sequences $\oar{r}$ with $r_1=3$ and $r_i\leq 5$ are
\begin{enumerate}
\item $(3\hd 3)$ for $d\geq 3$,
\item $(3,4,4\hd 4)$ for $d\geq 4$,
\item $(3,5,5\hd 5)$ for $d\geq 4$,
\item $(3,4,5\hd 5)$ for $d\geq 6$.
\end{enumerate}
\end{lemma}

 Lemma \ref{admisseqlem} is proved below. We first show how it implies the results
above, by showing in each case a general  $\phi$ satisfying the hypotheses of the theorem
must be in the appropriate $SFlat_{\oar r}^0$ with $\oar r$ admissible.
Since $\s_r(v_d(\BP V))$ is irreducible, one concludes.

\begin{proof}[Proof of   (\ref{s3vdpnthm})] We may assume that $\trank(\phi_{1,d-1})=3$, 
otherwise we are in the case of two variables.
By inheritance, it is sufficient to prove the result for $n=2$.
In this case, 
if $\trank(\phi_{a,d-a})\le 3$ for   $a=\dtwo$,
then  $\trank(\phi_{a,d-a})\le 3$ for all $a$ such that $2\le a\le d-2$,
and   $\oar r=(3\hd 3)$ is admissible .
\end{proof}

\begin{proof}[Proof of    (\ref{s45vdp2thm})]
  We may assume that $\trank\phi_{1,d-1}=3$. 
Let $d=4$.   If $\trank(\phi_{2,d-2})\le 4$ then by Lemma \ref{admisseqlem}
the only possible $\oar{r}$ are $(3,3)$ and $(3,4)$.
The first case corresponds to $\sigma_3(v_4(\P 2))\subset\sigma_4(v_4(\P 2))$
and the second case is as desired.
The case $d\ge 5$ is analogous.
\end{proof}

\begin{proof}[Proof of    (\ref{s5vdp2thm})]
   If $\phi\not\in\s_4(v_d(\BP W))$, then 
 $\phi\in Gor(3,5\hd 5)=\s_5(v_d(\pp 2))$ or $\phi \in Gor(3,4,5\hd 5)$.
 
It remains to show $Gor(3,4,5\hd 5)\subset \s_5(v_d(\pp 2))$. To prove this
we generalize the argument of \cite[Ex. 5.72]{MR1735271}.

Let $\phi \in SFlat_{3,4,5\hd 5}^0(S^dW)$ and consider the kernel of $\phi_{2,d-2}$ as
  a $2$-dimensional space of plane conics. If the base locus is a zero-dimensional scheme
of length four, then $\phi\in\sigma_4(v_d(\P 2))$, and $\trank(\phi_{3,d-3})\le 4$
which contradicts $\phi \in SFlat_{3,4,5\hd 5}^0(S^dW)$.
 Thus the base locus has dimension one and, in convenient coordinates,
the kernel of $\phi_{2,d-2}$ is   $<xy,xz>$. It follows that 
$\phi=x^d+\psi(y,z)$ and from $\trank(\phi_{a,d-a})=5$
it follows $\trank(\psi_{a,d-a})= 4$.  
 Since $\psi$ has two variables, 
this implies  $\psi\in \sigma_4(v_d(\P 2))$ 
and  
$\phi\in \sigma_5(v_d(\P 2))$ as required. 
\end{proof}

\begin{proof}[Proof of    (\ref{s6vdp2thm})]  
If the sequence of ranks of $\phi_{a,d-a}$  is $\oar r=(1,3,6,\ldots,6,3,1)$
then $\phi\in Gor(\oar r)=\sigma_6(v_d(\P n))$.

Otherwise $\trank(\phi_{2,d-2})\le 5$.
In this case, by an extension of Lemma \ref{admisseqlem} ,
 there are just   two other  possibilities for $\oar r$, namely
$\oar r_1=(1,3,5,6\ldots ,6,5,3,1)$
and $\oar r_2=(1,3,4,5,6\ldots ,6,5,4,3,1)$ for $d\ge 8$.

Consider $\phi\in \oar r_2$. As   above, we may assume that 
 $\tker(\phi_{2,d-2})=<xy,xz>$. It follows that 
$\phi=x^d+\psi(y,z)$. If if $d\le 9$ then  $\psi\in \sigma_5(v_d(\P 2))$.
If $d\ge 10$,     
  $\trank(\psi_{5,d-5})= 5$ because  $\trank(\phi_{5,d-5})=6$. Since $\psi$ has two variables, 
  $\psi\in \sigma_5(v_d(\P 2))$. Thus
$\phi\in \sigma_6(v_d(\P 2))$ as desired.

Consider $\phi\in \oar r_1$. The cases
 $d=6$  and $d\geq 8$ respectively follow  from \cite[Ex. 5.72]{MR1735271}
 and \cite[Thm. 5.71 (i)]{MR1735271}.
Here is a uniform argument for all $d\geq 6$: There is a conic $C$ in the kernel
of $\phi_{2,d-2}$. The four dimensional space of cubics which is the kernel of
$\phi_{3,d-3}$ is generated by $C\cdot L$ where $L$ is any line and a cubic $F$.
Then the $6$ points of $C\cap F$ are  the base locus,
by \cite[Cor 5.69]{MR1735271}, so   
 $\phi\in\sigma_6(v_d(\P 2))$.
\end{proof}

\medskip

Let $\tdim W=3$. We will need the following facts to prove  Lemma \ref{admisseqlem}:
\begin{enumerate}\label{dieselfacts}

\item For $\phi\in S^dW$, consider the ideal $\phi\upperp$ generated in degrees $\leq d$,  by  $\tker(\phi_{a,d-a})\subset S^aW^*$,
$1\leq a\leq d$, and the ring $A_{\phi}:=Sym(W^*)/\phi\upperp$. 
Note that the values  of the Hilbert function of $A_{\phi}$,  
$H_{A_{\phi}}(j)$,
are $T(\oar r):=(1,r_1\hd r_{\llcorner\frac d2\lrcorner},r_{\llcorner\frac d2\lrcorner}\hd r_1,1,0\hd 0)$, where
recall that $r_j=\trank(\phi_{j,d-j})$.

\item 
  \cite[Thm. 2.1]{MR0453723} The  ring $A_{\phi}$  has a minimal free resolution of the form
\begin{align}\label{BE}
0&\rig{}Sym(W^*)(-d-3)\rig{g^*}\oplus_{i=1}^uSym(W^*)(-{d+3-q_i})\\
& \nonumber \ \ \ \rig{h}\oplus_{i=1}^uSym(W^*)(-q_i)\rig{g}Sym(W^*)(0) \ra A_{\phi}\ra 0
\end{align}
where the $q_j$ are non-decreasing. 
Here $Sym(W^*)(j)$ denotes $Sym(W^*)$ with the labeling of degrees shifted by $j$,
so   $\tdim (Sym(W^*)(j))_d=\binom{d+j+2}{2}$.

\item 
Moreover \cite[Thm 1.1]{diesel} there is a unique   resolution with the properties
$q_1\leq\frac{d+3}2$, $q_i+q_{u-i+2}=d+2$, $2\leq i\leq \frac{u+1}2$  having $T(\oar r)$ as the values of its Hilbert function.   

\item Recall that $H_{A_{\phi}}(j)$ is also the alternating sum of the dimensions of the degree $j$
term in \eqref{BE} (forgetting the last term). Thus the $q_j$ determine the $r_i$.

\item \cite[Thm. 3.3]{diesel} Letting $j_0$ be the smallest $j$ such that $\phi_{j,d-j}$ is not injective, then $u=2j_0+1$ in the resolution
above.

\end{enumerate}

\begin{remark} Although we do not need these facts from \cite{MR0453723} here, we note that above: 
  $h$ is skew-symmetric,
$g$ is defined by the principal sub-Pfaffians of $h$, $A_{\phi}$ is an Artinian Gorenstein graded $\BC$-algebra
and 
 any Artinian Gorenstein graded 
$\BC$-algebra is isomorphic to $A_\phi$ for a polynomial $\phi$ uniquely determined up to constants.
The ring $A_{\phi}$ is called the {\it apolar ring} of $\phi$, the resolution above
is called {\it saturated}.  
\end{remark}

\begin{proof}[Proof of Lemma \ref{admisseqlem}] 
 Let $\oar r$ be a sequence as in the assumptions. Since $\tdim S^2\BC^3=6>5$, $u=5$ by Fact \ref{dieselfacts}(5).
Consider the unique   resolution satisfying  Fact \ref{dieselfacts}(3) having $T(\oar r)$ as Hilbert function.
The number of generators in degree $2$ is the number of the  $q_i$ equal  to $2$.
We must have $q_1=2$, otherwise, by computing the Hilbert function, $H(A_{\phi})(2)=6$ contrary to  assumption.  
The maximal number of generators in degree $2$ is three, otherwise we would have
$2+2=q_4+q_{5-4+2}= d+2$, a contradiction. 

If there are three generators in degree $2$ then $q_i=(2,2,2,d,d)$,
  and computing the Hilbert function via the Euler characteristic, we are in case (1)
  of the Lemma.

If there are two generators in degree $2$
then there are the two possibilities $q_i=(2,2,3,n-1,n)$ (case(2))
or $q_i=(2,2,4,n-2,n)$ (case (4)). Note that $q_i=(2,2,5,n-3,n)$ for $n\ge 8$   is impossible because 
otherwise
  $\oar r=(1,3,4,5,6,\ldots )$ contrary to  assumption. By similar arguments
one shows that  there are no other possibilities.

If there is one generator in degree $2$ then
$q_i=(2,3,3,n-1,n-1 )$ (case (3)).
\end{proof}

  \begin{proposition}\label{boundeven}
Let $\tdim V=3$ and $p\ge 2$. If the variety $\s_{r}(v_{2p}(\BP V))$ 
 is an irreducible component of
 $\s_{r}(v_2(\BP S^pV))\cap \BP S^{2p}V$
 then  $r\le  \frac{1}{2}p(p+1)$ or $(p,r)=(2,4)$ or $(3,9)$.

 \end{proposition}
\begin{proof}
Recall that for $(p,r)\neq (2,5)$ 
 \begin{equation}
  \tcodim \s_{r}(v_{2p}(\BP V))= {{2p+2}\choose 2}-3r
\end{equation}
and from Proposition \ref{segre} every irreducible component $X$ of  $\s_{r}(v_2(\BP S^pV))\cap \BP S^{2p}V$
satisfies

\begin{equation} 
 \tcodim X \le
\frac{1}{2}\left[{{p+2}\choose{2}}-r\right]\left[{{p+2}\choose{2}}-r+1\right]
\end{equation}
The result follows  by solving the inequality.
\end{proof}

For the case  $(p,r)=(2,4)$ 
see   Theorem \ref{bigflatthm}(\ref{s45vdp2thm}) and for the case $(p,r)=(3,9)$
see   Theorem \ref{sextics}.

\begin{corollary}\label{symflat}
$\deg \s_{{{p+1}\choose 2}}(v_{2p}(\pp 2))\le  \prod_{i=0}^{p}\frac{{{(p+2)(p+1)/2+i}\choose{p-i+1}}}
{{{2i+1}\choose{i}}}$.

 If   equality holds then  
 $\s_{{{p+1}\choose 2}}(v_{2p}(\pp 2))$   
 is given scheme-theoretically by the $\left(\frac{1}{2}\right)(p^2+p+2)$ minors of 
 the $(p,p)$-flattening.                       
\end{corollary}

 The values of the right hand side of the inequality for $p=1,\ldots 4 $ are respectively $4$, $112$,  $28314$, $81662152$.   

\begin{proof}
The right hand side is   
$\deg \s_{{{p+1}\choose 2}}(v_2(\BP^{p(p+3)/2}))$ by   Proposition \ref{segre}.
If $\phi\in\s_{{{p+1}\choose 2}}(v_{2p}(\pp 2))$ then $\trank(\phi_{p,p})\le
{{{p+1}\choose 2}}$. This means   $\s_{{{p+1}\choose 2}}(v_{2p}(\pp 2))$
is contained in a linear section of $\s_{{{p+1}\choose 2}}(v_2(\BP^{p(p+3)/2}))$,
and  by   Theorem \ref{maincata} it is a irreducible component of this linear section.
The result follows by the refined Bezout Theorem \ref{refinedbezout}.
\end{proof}

 Equality holds in Corollary \ref{symflat}  for  the cases $p=1, 2, 3$ by Proposition \ref{112}. 
The case $p=1$ corresponds to the quadratic Veronese surface and the cases $p=2, 3$ will be considered 
respectively in
Thm.\ref{bigflatthm} (\ref{s3vdpnthm}) and Theorem \ref{s6v6p2}.
For $p\ge 4$ these numbers are out of the range of the results in    \cite{MR1317230} (the points are too few 
to be fixed points of a torus action) and we do not know if   equality holds.

\section{Young flattenings for Veronese varieties}\label{yflatsect}

\subsection{Preliminaries}
 In what follows we fix $n+1=\tdim V$ and    endow $V$ with a volume form
 and thus identify (as $SL(V)$-modules)
 $S_{(p_1\hd p_{n+1})}V$ with $S_{(p_1-p_{n+1},p_2-p_{n+1}\hd p_{n}-p_{n+1},0)}V$.
 We   will say $(p_1-p_n,p_2-p_n\hd p_{n-1}-p_n,0)$ is the
{\it reduced  partition} associated to $(p_1\hd p_n)$.
  
 \smallskip

 The {\it Pieri formula} states that 
 $S_{\pi}V^*\subset S_{\nu}V^*\ot S^dV^*$ iff the Young diagram
 of 
 $\pi$ is obtained by adding $d$ boxes to the Young diagram
 of $\nu$, with no two boxes added to the same column. Moreover, 
   if this occurs, the multiplicity of $S_{\pi}V^*$ in
 $S_{\nu}V^*\ot S^dV^*$ is one.

 \smallskip
 
 We record the basic observation that the dual
 $SL(V)$-module to $S_{\pi}V$ is obtained
 by considering the complement to $\pi$ in the
 $\ell(\pi)\times \tdim V$ rectangle and rotating it
 to give a Young diagram whose associated partition
we denote $\pi^*$.

\smallskip
   
 Say $S_{\pi}V^*\subset S_{\nu}V\ot S^dV^*$ and consider the    map
 $S^dV\ra S_{\pi}V \ot S_{\nu}V^*$.
Let $S_{\mu}V=S_{\nu}V^*$ where $\mu$ is the reduced partition
with this property.
We obtain an inclusion $S^dV\ra S_{\pi}V \ot S_{\mu}V$.

 Given
 $\phi\in S^dV$,   let
 $\phi_{\pi,\mu} \in S_{\pi}V \ot S_{\mu}V$ denote the corresponding
 element.  If $S_{\mu}V=S_{\nu}V^*$ as an $SL(V)$-module, we will
also write $\phi_{\pi,\nu^*}=\phi_{\pi,\mu}$ when we consider it
as a linear map $S_{\nu}V\ra S_{\pi}V$.

{The following proposition is an immediate consequence of the subadditivity for ranks of linear maps.}

\begin{proposition}\label{mosect}
  Rank conditions on $\phi_{\pi,\mu}$  
 provide equations for the secant varieties of $v_d(\BP V)$ 
 as follows: Let $[x^d]\in v_d(\BP V)$.
 Say $\trank(x^d_{\pi,\mu})=t$. 
If $[\phi]\in \s_r(v_d(\BP V))$, then
$
\trank(\phi_{\pi,\mu})\leq rt
$.
Thus if $r+1\leq \tmin\{ \tdim S_{\pi}V,\tdim S_{\mu}V\}$,  the $(rt+1)\times (rt+1)$ minors of $\phi_{\pi,\mu}$ provide equations for
$\s_r(v_d(\BP V))$, i.e.,
$$
\La{rt+1}(S_{\pi}V^*)\ot \La{rt+1}(S_{\nu}V^*) \subset I_{rt+1}(\s_r(v_d(\BP V))).
$$
\end{proposition} 
  

\begin{remark}

 From the inclusion $S^dV\subset V\ot S^aV\ot S^{d-a-1}V$ we obtain  $v_d(\BP V)\subset Seg(\BP V\times \BP S^aV\times \BP S^{d-a-1}V)$.
  Strassen's equations for
$\s_{n+s}(Seg(\pp 2\times \pp{n-1}\times \pp{n-1}))$
give equations for secant varieties of Veronese varieties via this three-way
flattening. The Aronhold equation   comes from the Strassen equations  for $\s_3(Seg(\pp 2\times \pp 2\times\pp 2))$.


So far we have not obtained any new equations using three way symmetric flattenings and
Young flattenings.
We mention three-way symmetric flattenings because they  may be useful in future investigations, especially
when further modules of equations for
secant varieties of  triple Segre products are found.
\end{remark}

Let
$$
YFlat_{\pi,\mu}^t(S^dV):=\ol{
\{ \phi\in S^dV\mid \trank(\phi_{\pi,\mu})\leq t\} }
=
\s_t(Seg(\BP S_{\pi}V\times \BP S_{\mu}V))\cap \BP S^dV.
$$

\begin{remark} Recall  $YF_{d,n}^r$ from \S\ref{YFsect} whose
description had redundancies. We can now give an irredundant description
of its defining equations:
$$
YF_{d,n}^r=YFlat_{((\d+1)^a,\d^{n-a}), (\d+1,1^a)}^{\binom n{\lfloor \frac n2\rfloor} r}(S^dV).
$$

{This is a consequence of Schur's Lemma, because the module $S_{(\d+1,1^a)}V$ is the only one appearing in both sides of
$S^dV\otimes S^{\d}V^*\otimes\wedge^aV\to S^{\d}V\otimes\wedge^{a+1}V$}
\end{remark}

Note that if $S_{\pi}V\simeq S_{\mu}V$ as $SL(V)$-modules and the map is
symmetric, then
$$
YFlat_{\pi,\mu}^t(S^dV)=\s_t(v_2(\BP S_{\pi}V))\cap \BP S^dV,
$$
and if it is skew-symmetric, then 
$$
YFlat_{\pi,\mu}^t(S^dV)=\s_t(
G(2, S_{\pi}V))\cap \BP S^dV.
$$

 \subsection{The surface case, $\s_r(v_d(\pp 2))$}\label{surfacecasesect}
 
 In this subsection fix $\tdim V=3$ and a volume form $\Omega$ on $V$.
 From the general  formula for $\tdim S_{\pi}V$ (see, e.g., \cite[p78]{FH}), we record
 the special case:
 \be\label{dimsab}
 \tdim S_{a,b}\BC^3=\frac 12(a+2)(b+1)(a-b+1).
\ene

 \begin{lemma}\label{x3ranklem} Let $a\geq b$. Write $d=\a+\b+\g$ with $\alpha\leq b$,
 $\b\leq a-b$ so  $S_{(a+\g-\a,b+\b-\a)}V\subset
 S_{a,b}V\ot S^dV$. For $\phi\in S^dV$, consider the
 induced map
\be\label{inducedmap}
 \phi_{(a,b),(a+\g-\a,b+\b-\a)}: S_{a,b}V^*\ra S_{(a+\g-\a,b+\b-\a)}V.
\ene 
 Let $x\in V$, then 
 \be\label{rankxcube}\trank((x^d)_{(a,b),(a+\g-\a,b+\b-\a)})=
 \frac 12(b-\a+1)(a-b-\b+1)(a+\b-\a+2)=: R .
 \ene
Thus in this situation $\La {pR+1}(S_{ab}V)\ot \La {pR+1}(S_{a+\g-\a,b+\b-\a}V)$ gives
nontrivial degree $pR+1$ equations for $\s_p(v_d(\pp 2))$.
 \end{lemma}

\begin{remark} Note that the right hand sides of  equations \eqref{dimsab} and \eqref{rankxcube} 
  are the same when $\a=\b=0$. To get useful equations one  wants $R$ small with respect to $\tdim S_{a,b}\BC^3$.  
\end{remark}

 \begin{proof}
In the following picture we   label  the first row containing $a$ boxes with $a$ and so on.

$$\young(aaaaa,bbb)\quad \to\quad \young(aaaaa\gamma\gamma,bbb\b,\a\a)$$

 Assume we have chosen a weight basis $x_1,x_2,x_3$  of $V$
 and $x=x_3$ is a vector of lowest weight.
 Consider the image of a weight basis of $S_{a,b}V$ under
 $(x^3_3)_{(a,b),(a+\g-\a,b+\b-\a)}$. Namely   consider all semi-standard
 fillings of the Young diagram corresponding to $(a,b)$, and count how many
 do not map to zero. By construction, the images of all the vectors that do
 not map to zero are linearly independent, so this count indeed gives the dimension
 of the image.

 In order to have a vector not in the kernel,  
 the first $\a$ boxes of the first row must be filled with
 $1$'s and the first $\a$ boxes of the second row must
 be filled with $2$'s. 
 
 Consider the next $(b-\a,b-\a)$ subdiagram. Let $\BC^2_{ij}$   denote the span of $x_i,x_j$
 and $\BC^1_i$ the span of $x_i$.
The boxes
 here can be filled with any semi-standard filling using
 $1$'s $2$'s and $3$'s, but the freedom to fill the rest will
 depend on the nature of the filling, so   split
 $S_{b-\a,b-\a}\BC^3$ into two parts, the first part where
 the entry in the last box in the first row is $1$, which
 has $\tdim S_{b-\a}\BC^2_{23}$ fillings and
 second where the entry in the last box in the first
 row is $2$, which has $(\tdim S_{b-\a,b-\a}\BC^3- \tdim S_{b-\a}\BC^2_{23})$
 fillings.
 
 In the first case, the next $\b$-entries are free
 to be any semi-standard   row filled with 
 $1$'s and $2$'s, of which there are $\tdim S_{\b}\BC^2_{12}$
 in number, but we split this further into two sub-cases, depending on
 whether the last entry is a $1$ (of which there is one ($=\tdim S_{\b}\BC^1_1$) such),
 or $2$ (of which there are $(\tdim S_{\b}\BC^2_{12}-\tdim S_{\b}\BC^1_1)$ such). In the first sub-case
 the last $a-(b+\b)$ entries admit $\tdim S_{a-(b+\b)}\BC^3$ fillings
 and in the second sub-case there are $\tdim S_{a-(b+\b)}\BC^2_{23}$ such. 
 Putting these together, the total number of fillings 
 for the various paths corresponding to the first part is
\begin{align*}&
 \tdim S_{b-\a}\BC^2_{12}
 [(\tdim S_{\b}\BC^1_1)(\tdim S_{a-b-\b}\BC^3)
 +(\tdim S_{\b}\BC^2_{12}-\tdim S_{\b}\BC^1_1)(\tdim S_{a-b-\b}\BC^2_{23})]
 \\
 &= (b-\a+1)[(1)\binom{a-b-\b+2}2 +(\b-1)(a-b-\b+1)].
 \end{align*}

For the second part,  the next $\b$ boxes in the first
row must be filled with $2$'s (giving $1=\tdim S_{\b}\BC^1_2$) and the last
$a-(b+\b)$ boxes can be filled with $2$'s or $3$'s semi-standardly,
i.e., there are $\tdim S_{a-b-\b}\BC^2_{23}$'s worth.
So the contribution of the second part is
\begin{align*}&
(\tdim S_{b-\a,b-\a}\BC^3-\tdim S_{b-\a}\BC^2_{12})
(\tdim S_{\b}\BC^1_2)(\tdim S_{a-b-\b}\BC^2_{23})\\
&=
(\binom{b-\a+2}{b-\a}-(b-\a+1))(1)(a-b-\b+1).
\end{align*}
Adding up gives the result.
 \end{proof}

We are particularly interested in cases where $(a,b)=(a+\g-\a,b+\b-\a)$. 
In this case  
\begin{align}
\a&=\g= \frac 13 {(d+2b-a)}\\
\b&=\frac 13(d-4b+2a).
\end{align}
Plugging into the conclusion of Lemma \ref{x3ranklem}, the  
rank of the image of a $d$-th power in this situation is
$$
\frac 19 (a+b-d+3)^2 {(a-b+1)}.
$$
To keep this small, it is convenient to take $d=a+b$ so the rank is
$a-b+1$. One can then fix this number and let $a,b$ grow to study
series of cases.

If \eqref{rankxcube} has  rank one  when $d=2p$, we   just recover the
usual symmetric flattenings as $S_{(p,p)}V=S_pV^*$.
We consider the next two cases in the theorems below, $(a,b)=(p+1,p)$ when 
$d=2p+1$ and $(a,b)=(p+2,p)$ when $d=2p+2$. Recall
that $(p+q,p)^*=(p+q,q)$ in the notation of \S\ref{mosect}.
 
 \medskip

 
Let $d=2p+1$.
The skew analog of   Proposition \ref{boundeven} is the following proposition,
which shows that the bound in the assumption of  
 Theorem \ref{mainyoung} is sharp.

\begin{proposition}\label{boundodd}
Let $\dim V=3$. If the variety $\sigma_r(v_{2p+1}(\pp {{}}V))$ is an irreducible component
of $YFlat^r_{(p+1,p),(p+1,p)}(S^{2p+1}V)$, 
then $r\le{{p+2}\choose 2}$.
\end{proposition}

\begin{proof}

Recall from Proposition \ref{segre}
 \begin{align}
  \tcodim \s_{r}(v_{2p+1}(\BP V))&= {{2p+3}\choose 2}-3r
\\ 
 \tcodim \sigma_r(G(2,S_{p+1,p}V)\cap \pp {{}}S^{2p+1}V& \le
\frac{1}{2}\left[(p+1)(p+3)-2r\right]\left[(p+1)(p+3)-2r-1\right]
\end{align}
The inequality is a consequence of $\tcodim \s_{r}(v_{2p+1}(\BP V))\leq
 \tcodim \sigma_r(G(2,S_{p+1,p}V)\cap \pp {{}}S^{2p+1}V$.
\end{proof}

Here is a pictorial description  when $p=2$ of
$\phi_{31,31}\colon S_{3 2}V\to S_{3 1}V$
  in terms of Young diagrams:
$$\yng(3,2)\otimes\young(*****)\quad\to\quad\young(\ \ \ **,\ \ *,**)\simeq\ \yng(3,1)$$

\begin{corollary}\label{degequal}
$\deg \s_{{{p+2}\choose 2}}(v_{2p+1}(\pp 2))\le  \frac{1}{2^p}\prod_{i=0}^{p-1}\frac{{{(p+3)(p+1)+i}\choose{p-i}}}{{{2i+1}\choose{i}}}$.

 If   equality holds then  
 $\s_{{{p+2}\choose 2}}(v_{2p+1}(\pp 2))$   
 is given scheme-theoretically by the size $(p+2)(p+1)+2$ sub-Pfaffians of 
 $\phi_{(p+1,p),(p+1,p)}$.   
\end{corollary}

 The values of the right hand side for $p=1,\ldots 4 $ are respectively $4$, $140$, 
 $65780$, $563178924$.

\begin{proof}
The right hand side is
$\deg \s_{{{p+2}\choose 2}}(G(\BC^2,S_{(p+1,p)}V)$ by   Proposition \ref{segre}.
Now   $\s_{{{p+1}\choose 2}}(v_{2p+1}(\pp 2))$
is contained in a linear section of $ \s_{{{p+2}\choose 2}}(G(\BC^2,S_{(p+1,p)}V))$,
and  by  Theorem \ref{mainyoung} it is a irreducible component of this linear section.
The result follows by the refined Bezout theorem  \ref{refinedbezout}.
\end{proof}

In Corollary \ref{degequal} equality holds in  the cases $p=1, 2$  
by Proposition \ref{112}. The case $p=1$ is just the Aronhold case and the case $p=2$ will be considered in
Theorem \ref{s6v5p2} (2).
For $p\ge 3$ these numbers are out of the range of  \cite{MR1317230} and we do not know if   equality holds.

Note that the usual symmetric
flattenings only give equations for  $\s_{k-1}(v_{2p+1}(\pp 2))$for $k\leq \frac 12(p^2+3p+2)$.

\medskip
Now let $d=2p+2$ be even, requiring $\pi=\mu$,
  the smallest possible $\trank((x^d)_{\pi,\mu})$ is three, which we obtain with
$\phi_{(p+2,p),(p+2,p)}$.

\begin{proposition}\label{pp2pprop}
Let $d=2p+2$.
The Young flattening $\phi_{(p+2,2),(p+2,2)}\in S_{p+2,2}V\ot S_{p+2,2}V$  is
symmetric. It is
of rank three for $\phi\in v_d(\pp 2)$ and gives degree $3(k+1)$ equations for
$\s_{r}(v_{2p+2}(\pp 2))$ for $r\leq \frac 12(p^2+5p+4)-1$.
A convenient model for the equations is given in the proof.
\end{proposition}

A pictorial description when $p=2$ is as follows:
$$\yng(4,2)\otimes\young(******)\quad\to\quad\young(\ \ \ \ **,\ \ **,**)\simeq\ \yng(4,2).
$$

\begin{proof} Let $\tilde\Omega\in \La 3V$ be dual to the volume form $\Omega$.
To prove the symmetry, for $\phi=x^{2p+2}$,   consider the map,
\begin{align*}
M_{x^{2p+2}}: S^{p}V^*\ot S^2(\La 2 V^*)&\ra S^pV\ot S^2(\La 2 V)\\
\a_1\cdots \a_p   \ot (\g_1\ww\d_1)\circ (\g_2\ww\d_2) &\mapsto
\a_1(x)\cdots \a_p(x) x^p\ot \tilde\Omega(x\intprod \g_1\ww \d_1)\circ \tilde\Omega(x\intprod \g_2\ww \d_2)
\end{align*}
and define $M_{\phi}$ for arbitrary $\phi\in S^{2p+2}V$ by linearity and polarization.
If we take bases of $S^2V\ot S^2(\La 2 V)$ as above, with indices $((i_1\hd i_p),(kl),(k'l'))$, 
most of the   matrix of $M_{e_1^{2p+2}}$ is zero. The upper-right hand
$6\times 6$ block, where $(i_1\hd i_p)=(1\hd 1)$ in both rows and columns and 
the order on the other indices 
$$((12),(12)),((13),(13)),((12),(13)),((12),(23)) ,((13),(23)),((23),(23)),
$$
 is
$$
\begin{pmatrix}
0&1&0&0&0&0\\1& 0&0&0&0&0\\ 0&0&1&0&0&0\\ 0&0&0&0&0&0\\0&0&0&0&0&0\\0&0&0&0&0&0 
\end{pmatrix}
$$
showing the symmetry.
Now
\begin{align*}
(S^pV\ot S^2(\La 2 V))^{\ot 2}&= 
(S^pV\ot (S_{22}V\op S_{1111}V)^{\ot 2}\\
&=S_{p+2,2}V\op stuff 
\end{align*}
where all the terms in $stuff$ have partitions with at least three parts.
On the other hand, from the nature of the image we conclude it is just the first factor  
and
$M_{\phi}\in S^2(S_{p+2,2}V)
$.
\end{proof}

\medskip
 


Note that the usual symmetric flattenings give nontrivial equations 
for
$\s_{k-1}(v_{2p+2}(\pp 2))$ for $k\leq \frac 12(p^2+5p+6)$, a larger range than
in Proposition \ref{pp2pprop}. However we show (Theorem \ref{sextics}) that the symmetric flattenings alone are not
enough to cut out $\s_7(v_6(\pp 2))$,  but with  the $((p+1,2),(p+1,2))$-Young
flattening they are.

Here is a more general Young flattening:

\begin{proposition} Let $d=p+4q-1$. The Young flattening 
$$\phi_{(p+2q ,2q-1 ),(p+2q ,2q-1 )}\in S_{(p+2q ,2q-1 )}V\ot S_{(p+2q ,2q-1 )}V,
$$
  is
  skew-symmetric if $p$ is even and symmetric if $p$ is odd. 

Since it has rank $p$ if $\phi\in v_d(\pp 2)$, if $p$ is even (resp. odd),  the size $kp+2$ sub-Pfaffians 
(resp. size $kp+1$ minors) of 
$\phi_{(p+2q,2q-1),(p+2q,2q-1)}$ give degree $\frac{kp}2+1$ (resp. $kp+1$) equations for
$\s_{k}(v_{p+4q-1}(\pp 2))$ for 
$$k\leq \frac {q(p+2q+2)(p+2)}{p}
.
$$
\end{proposition}

\begin{proof}
Consider $M_{\phi} : S^{p-1}V^*\ot S^q(\La 2 V^*)\ra S^pV \ot S^q(\La 2 V )$
given for $\phi=x^{p+4q-1}$ by
$$
\a_1\cdots \a_{p-1}\ot \b_1\ww \g_1\cdots \b_q\ww\g_q
\mapsto
\Pi_j (\a_j(x))x^{p-1}\ot \tilde\Omega(x\intprod \b_1\ww \g_1)\cdots \tilde\Omega(x\intprod \b_q\ww \g_q)
$$
and argue as above.
\end{proof}

Here the usual flattenings give degree $k$ equations for $\s_{k-1}(v_d(\pp 2))$ in the generally
larger range $k\leq \frac 18(p+4q+2)(p+4q)$.

Recall that we have already determined the ideals  of  $\s_r(v_d(\pp 2))$ for $d\le 4$,
(Thm. \ref{bigflatthm} and   the chart in the introduction) so we next consider 
the case $d=5$.

 {\it Case $d=5$}:
The symmetric flattening given by the size $(k+1)$ minors of
$\phi_{2,3}$ define $\s_k(v_5(\pp 2))$ up to $k=4$ by Theorem \ref{maincata} (3).
Note that the size $6$ minors define a subvariety of codimension $5$, strictly containing 
$\s_5(v_5(\pp 2))$, which has codimension $6$. So, in this case, the bound
provided by Theorem \ref{maincata} is sharp.

\begin{theorem}\label{s6v5p2} \

\begin{enumerate}
\item $\s_k(v_5(\pp 2))$ for $k\le 5$ is  an irreducible component of $YFlat^{2k}_{31,31}(S^5\BC^3)$,  the variety given by the principal
size $2k+2$ Pfaffians  of the $[(31),(31)]$-Young flattenings.

\item  the principal
size $14$ Pfaffians 
of the $[(31),(31)]$-Young flattenings are scheme-theoretic defining equations for $\s_6(v_5(\pp 2))$, i.e., 
as schemes, $\s_6(v_5(\pp 2))= YFlat^{12}_{31,31}(S^5\BC^3)$. 
\item $\s_7(v_5(\pp 2))$ is the ambient space.
\end{enumerate}
 \end{theorem}

\begin{proof} 
(1) is a consequence of Theorem \ref{mainyoung}.
 
To prove the second assertion,   by the Segre formula (Proposition \ref{segre}) the subvariety of
size $15$ skew-symmetric  
matrices of rank $\le 12$ has codimension $3$ and degree $140$. By  Proposition \ref{112},  $\s_6(v_5(\pp 2))$ 
has codimension $3$ and degree $140$. We conclude  as in the proof of Corollary \ref{symflat}.
\end{proof}

 {\it Case $d=6$}:
In the case $\s_k(v_6(\pp 2))$,
the symmetric flattening given by the $(k+1)$ minors of
$\phi_{2,4}$ define $\s_k(v_6(\pp 2))$ as an irreducible component up to $k=4$
(in the case $k=4$ S. Diesel  \cite[Ex. 3.6]{MR1735271}  showed that there are two irreducible components)
and the  symmetric flattening given by the $(k+1)$ minors of
$\phi_{3,3}$ define $\s_k(v_6(\pp 2))$ as an irreducible component up to $k=6$.

  The following is a special case of the Thm. 
\ref{bigflatthm}(\ref{s6vdp2thm}), we include this second proof because it is very short.

\begin{theorem}\label{s6v6p2} As schemes,  $\s_6(v_6(\pp 2))=Rank_{3,3}^6(S^6\BC^3)$, i.e., 
the size $7$ minors of $\phi_{3,3}$ cut out  $\s_6(v_6(\pp 2))$ scheme-theoretically.
\end{theorem}
\begin{proof}
 By the Segre formula (Proposition \ref{segre}) the subvariety of symmetric $10\times 10$
matrices of rank $\le 6$ has codimension $10$ and degree $28,314$. By Proposition \ref{112},  $\s_6(v_6(\pp 2))$ 
has  codimension $10$ and degree $28,314$. We conclude  as in the proof of  Corollary \ref{symflat}.
\end{proof}

The size $8$ minors of $\phi_{3,3}$ define a subvariety of codimension $6$, strictly containing
$\s_7(v_6(\pp 2))$, which has codimension $7$. In the same way, the size $9$ minors of $\phi_{3,3}$ define a subvariety 
of codimension $3$, strictly containing
$\s_8(v_6(\pp 2))$, which has codimension $4$.
 Below we construct   equations in terms of Young flattenings.

$\det (\phi_{42,42})$ is a polynomial of degree 27, which is not the power of a lower degree polynomial. 
This can be proved by cutting with a random projective line, and
using Macaulay2. The two variable  polynomial obtained is not the power of a lower degree polynomial.

If $\phi$ is decomposable then $\trank (\phi_{42,42})=3$ by Lemma \ref{x3ranklem},
so that when $\phi\in \s_k(v_6(\pp 2))$ then $\trank (\phi_{42,42})\le 3k$.

\begin{theorem}\label{sextics} \
\begin{enumerate}
\item  $\s_7(v_6(\pp 2))$ is an irreducible component of $Rank_{3,3}^{7}(S^6\BC^3)\cap YFlat_{42,42}^{22}(S^6\BC^3)$, i.e., of
 the variety defined by the size $8$ minors of the symmetric flattening $\phi_{3,3}$
and by the size  $22$ minors of the $[(42),(42)]$-Young flattenings.
 
\item  $\s_8(v_6(\pp 2))$ is an irreducible component of
$Rank_{3,3}^{8}(S^6\BC^3)\cap YFlat_{42,42}^{24}(S^6\BC^3)$, i.e., of
 the variety defined by the size $9$ minors of the symmetric flattening $\phi_{3,3}$
and by the  size $25$ minors of the $[(42),(42)]$-Young flattenings.

\item  $\s_9(v_6(\pp 2))$ is the hypersurface of degree $10$ defined by $\tdet (\phi_{3,3})$.
\end{enumerate}
\end{theorem}

\begin{proof}
  
To prove (2),  we   picked
a polynomial $\phi$  which is the sum of $8$ random fifth powers of linear forms, and
a  submatrix of $\phi_{42,42}$ of order $24$ which is invertible.
The matrix representing $\phi_{42,42}$ can be constructed explicitly by the package
PieriMaps of Macaulay2  \cite{M2,pierimaps}.

The affine tangent space at $\phi$ of $Rank_{3,3}^{8}(S^6\BC^3)$
has codimension $3$ in $S^6\BC^3$.
In order to compute a tangent space,   differentiate, as usual, each line of the matrix, substitute $\phi$
in the other lines, compute the determinant and then sum  over the lines.
It is enough to pick one minor of order $25$ of $\phi_{42,42}$ containing the invertible ones
of order $24$. The tangent space of this minor at $\phi$ is not contained in
the subspace of codimension $3$, yielding the desired subspace of codimension $4$.

(1) can be proved in the same way.
(3) is   well known.
\end{proof}

\begin{remark} 
$\s_7(v_6(\pp 2))$ is the first example where the known equations are not of minimal possible degree.
\end{remark}

{\it Case $d=7$}: 
In the case of  $\s_k(v_7(\pp 2))$,
the symmetric flattening given by the size $(k+1)$ minors of
$\phi_{3,4}$ define $\s_k(v_7(\pp 2))$ as an irreducible component up to $k=8$.  

\begin{theorem}\label{v7}\ 
\begin{enumerate}
\item  For $k\le 10$ $\s_{k}(v_7(\pp 2))$  is an irreducible component
of $YFlat_{41,41}^{k}(S^7\BC^3)$, which is defined by
   the size $(2k+2)$
subpfaffians of  
 of $\phi_{41,41}$.

\item   $\s_{11}(v_7(\pp 2))$ has codimension $3$ and it is contained in the hypersurface
$YFlat_{41,41}^{22}(S^7\BC^3)$ of degree $12$ 
defined by 
$\tpfaff (\phi_{41,41})$.
\end{enumerate}
\end{theorem}

\begin{proof}
(1) follows from   Theorem \ref{mainyoung}. (2) is obvious.
\end{proof}

\begin{remark}
We emphasize that $\s_{11}(v_7(\pp 2))$ is the first case where we do not know, even conjecturally,
further equations.
 \end{remark}

 {\it Case $d=8$}:  It is possible that 
  $\s_k(v_8(\pp 2))$ is an irreducible component of $Rank_{4,4}^{k+1}(S^8\BC^3)$
  up to $k=12$, but the bound given in Theorem \ref{maincata}
is just $k\le 10$.
The size $14$ minors of $\phi_{4,4}$ define a subvariety of codimension $4$,
which strictly contains $\s_{13}(v_8(\pp 2))$
which has codimension $6$. Other equations of degree $40$ for 
$\s_{13}(v_8(\pp 2))$ are given by the size $40$  minors of   $\phi_{53,52}$. It is possible that
minors of $\phi_{(53),(52)}$ could be used to get a collection of set-theoretic equations  for  $\s_{12}(v_8(\pp 2))$ and $\s_{11}(v_8(\pp 2))$. In the same way,
$\det\phi_{4,4}$ is just an equation of degree $15$ for $\s_{14}(v_8(\pp 2))$
which has codimension $3$.

\smallskip
 {\it Case $d=9$}: It is possible that 
 $\s_k(v_9(\pp 2))$ is an irreducible component of $Rank_{4,5}^{k+1}(S^9\BC^3)$
  up to $k=13$, but the bound given in Theorem \ref{maincata}
is just $k\le 11$.
 $Rank_{4,5}^{14}(S^9\BC^3)$
  strictly contains $\s_{14}(v_9(\pp 2))$
which has codimension $13$.
The principal Pfaffians of order $30$ of $\phi_{54,51}$ give further equations for $\s_{14}(v_9(\pp 2))$.
In the same way the principal Pfaffians of order $32$ (resp. $34$) of $\phi_{54,51}$ give some equations for $\s_{15}(v_9(\pp 2))$
(resp. $\s_{16}(v_9(\pp 2))$).
Another equation for $\s_{15}(v_9(\pp 2))$ is the Pfaffian
of the skew-symmetric morphism
$\phi_{63,63}\colon S_{6,3}V\to S_{6,3}V$
which can be pictorially described in terms of Young diagrams, by
$$\yng(6,3)\otimes\young(*********)\quad\to\quad\young(\ \ \ \ \ \ ***,\ \ \ ***,***)\simeq\ \yng(6,3)$$

We do not know any equations for  $\s_{k}(v_9(\pp 2))$ when $k=17, 18$.
The $k=18$ case is particularly interesting because the corresponding secant variety is a hypersurface.

  \section{Construction of equations from vector bundles}\label{constsect}

\subsection{Main observation}



Write $e:=\trank(E)$ and 
consider the {\it determinantal varieties} or {\it rank varieties} defined by the minors of $A^E_v:  H^0(E) \to H^0(E^{*}\otimes L)^{*}$  (defined by equation (\ref{pairingmap})),  
\begin{align*}Rank_k(E):&= \BP \{v\in V| \trank(A_{v}^E)\le ek\}\\
&
=\s_{ek}(Seg(\BP H^0(E)^{*}\times \BP H^0(E^*\ot L)^{*}))\cap \BP V.
\end{align*}

\begin{proposition}\label{construct} Let $X\subset \BP V=\BP H^0(L)^*$ be a variety, and $E$
a rank $e$   {vector bundle on $X$.}  Then 
$$\sigma_r(X)\subseteq Rank_r(E)
$$
i.e.,  the size $(r e+1)$  minors of $A_v^E$   give equations for $\sigma_r(X)$.
\end{proposition}

When $E$ is understood, we will write $A_v$ for $A_v^E$.

\begin{proof}

By \eqref{pairingmap}, if $x=[v]\in X$, then 
$H^0(I_{x}\otimes E)\subseteq \ker A_v$.  
 {The subspace $H^0(I_{x}\otimes E)\subseteq H^0(E)$ has codimension
at most $e$ in $H^0(E)$, hence the same is true for
the subspace $\ker A_x$ and it follows $\trank(A_{x}) \le e$.
If $v\in\hat \sigma_r(X)$ is a general point, then it may be expressed as  $v=\sum_{i=1}^rx_i$, with $x_i\in \hat X$.
Hence 
$$\trank(A_v) =\trank(\sum_{i=1}^rA_{x_i})\le
\sum_{i=1}^r\trank(A_{x_i})\le r e
$$ 
as claimed. Since the inequality  is a closed condition, it
holds for all $v\in\hat \sigma_r(X)$.}
\end{proof}

\begin{example}\label{cata} Let  $X=v_d(\BP W)$, $L=\cO(d)$, $E=\cO(a)$ so $E^*\ot L=\cO(d-a)$, then
$Rank_r(E)$ is the {\it catalecticant variety} of symmetric flattenings in
$S^{a}W\otimes S^{d-a}W$ of rank at most $r$.
\end{example}
\begin{example}     Let   $E$ be a   homogeneous 
bundle on $\pp n$ with   $H^0(E)^{*}=S_{\mu}V$, $H^0(E^{*}\otimes L)^{*}=S_{\pi}V$.
Then $A_v$ corresponds to $\phi_{\pi,\mu}$ of \S\ref{yflatsect}. 
\end{example}

  The generalization  of
  symmetric flattenings to   Young flattenings for Veronese varieties 
is a representation-theoretic version of  the generalization  from line bundles to higher rank vector bundles.
 
\begin{example}
 A general source of examples is given by curves obtained as determinantal loci.
This topic is studied in detail in \cite{MR944326}. In \cite{ginensky}, 
A. Ginensky  considers 
the secant varieties $\sigma_k(C)$ to smooth curves $C$ in their bicanonical embedding.
With our notations this corresponds to the symmetric pair $(E,L)=(K_C,K_C^2)$.
He proves   \cite[Thm. 2.1]{ginensky} that $\sigma_k(C)=Rank_k(K_C)$ if $k<\textrm{Cliff}(C)$
and $\sigma_k(C) \subsetneq  Rank_k(K_C)$ for larger $k$. Here $\textrm{Cliff}(C)$ denotes the Clifford index of $C$.
\end{example}

\medskip

\subsection{The construction in the symmetric and skew-symmetric cases}
We say that $(E,L)$ is a {\it symmetric pair} if the
 isomorphism $E\rig{\alpha} E^{*}\otimes L$ is symmetric,
that is the transpose isomorphism $E\otimes L^{*}\rig{\alpha^t}E^{*}$,
after tensoring by  $L$ and multiplying the map by $1_L$ equals $\alpha$,
i.e.,  $\alpha=\alpha^t\otimes 1_L$. In this case   
 $S^2E$ contains $L$ 
as a direct summand, the morphism $A_v$ is symmetric,
and $Rank_k(E)$ is defined by the $(ke+1)$-minors of $A_v$.

Similarly,  $(E,L)$ is a {\it skew-symmetric pair} if $\alpha=-\alpha^t\otimes 1_L$.
In this case $ e$ is even,
 $\wedge^2E$ contains $L$ 
as a direct summand, the morphism $A_v$ is skew-symmetric,
and $Rank_k(E)$ is defined by the 
size $(k e+2)$ subpfaffians of $A_v$, which are equations of degree $\frac{ke}2+1$.

  {
\begin{example} Let $X=\BP^2\times\BP^n$ embedded by $L=\cO(1,2)$.
The equations for $\sigma_k(X)$ recently considered in
\cite{lukesegver}, where they  have been called 
{\it exterior flattenings},  fit in this setting. Call $p_1, p_2$ the two projections.
Let $Q$ be the tautological quotient bundle on $\BP^2$ and let $E=p_1^*Q\otimes p_2^*\cO(1)$, then $(E,L)$ is a skew-symmetric pair which gives rise to  the equations (2) in Theorem 1.1
of \cite{lukesegver}, while the equations (1) are obtained with $E=p_2^*\cO(1)$.
\end{example}
}

\subsection{The conormal space}

The results reviewed in \S\ref{conormal}, restated in the language of vector bundles, say
the affine conormal space of $Rank_k(E)$ at $[v]\in Rank_k(E)_{smooth}$ is   the image of the map
$$
\ker A_{v}\otimes \textrm{Im\ }A_{v}^{\perp}\to H^0(L)=V^*.
$$

If $(E,L)$ is a symmetric (resp. skew symmetric) pair, there is a symmetric
(resp. skew symmetric) isomorphism $\ker A_{v}\simeq \textrm{Im\ }A_{v}^{\perp}$
and the conormal space of $Rank_k(E)$ at $v$ is given by the image of the map
$S^2\left(\ker A_{v}\right)\to H^0(L)$, (resp. $\wedge^2\left(\ker A_{v}\right)\to H^0(L)$).

\subsection{A sufficient criterion for $\s_k(X)$ to
be an irreducible component of $Rank_k(E)$}\label{critsect}

Let  $v\in \hat X$,     in the proof
of Proposition \ref{construct}, we saw   $H^0(I_{v}\otimes E)\subseteq \ker A_v$.
In the same way, 
$H^0(I_{v}\otimes E^{*}\otimes L)\subseteq \textrm{Im\ }A_{v}^{\perp}$, by taking transpose.
Equality holds if $E$ is spanned at $x=[v]$.
This is generalized by the following Proposition.

\begin{proposition}\label{kerdec}
Let $v=\sum_{i=1}^kx_i\in V$, with $[x_i]\in X$,  and let $Z=\{[x_1],\ldots , [x_k]\}$.
Then 
\begin{align*}H^0(I_{Z}\otimes E)&\subseteq \ker A_v\\
 H^0(I_{Z}\otimes E^{*}\otimes L)&\subseteq \textrm{Im\ }A_{v}^{\perp}.
\end{align*}
The first inclusion is an equality if 
$H^0(E^{*}\otimes L)\to H^0(E^{*}\otimes L_{|Z})$
is surjective.
The second inclusion is an equality if
$H^0(E)\to H^0(E_{|Z})$
is surjective. \end{proposition}

\begin{remark} In   \cite{BFS}, a line bundle $F\ra X$ is defined to be {\it $k$-spanned} if
$H^0(F)$ surjects onto $H^0(F|_Z)$ for all $Z=\{[x_1],\ldots , [x_k]\}$. In that paper
and in subsequent work they study which
line bundles have this property. In particular $k$-spanned-ness of $X\subset \BP H^0(F)^*$ implies
that for all $k$-tuples of  points $([x_1]\hd [x_k])$ on $X$, writing $Z=\{[x_1],\ldots , [x_k]\}$,
then  $\langle Z\rangle=\pp{k-1}$,
as for example occurs with Veronese varieties $v_d(\BP V)$ when $k\leq d+1$.
\end{remark}

\begin{proof} If $E^{*}\otimes L$ is spanned at {$x=[w]$} then 
we claim { $H^0(I_x\otimes E)=\ker A_w$.}
To see this, 
work over an open set where $E,L$ are trivializable and take   trivializations.
There are $t_j \in H^0(E^*\otimes L )$
such that in a  basis $e_i$ of  the $\BC^e$ which  
we identify with the fibers of $E$ on this open subset,  $\langle e_i,t_j(x)\rangle =\delta_{ij}$. 
Take  {$s\in \tker(A_w)$.} By assumption
$\langle s(x), t_j(x)\rangle =0$ for every $j$, hence,
 writing  $s=\sum s_ie_i$,  $s_j(x)=0$ for every $j$, i.e.,   
$s\in H^0(I_x\otimes E)$. 

  Since
$$H^0(I_{Z}\otimes E)=\cap_{i=1}^kH^0(I_{x_i}\otimes E)\subseteq \cap_{i=1}^k\ker A_{x_i}\subseteq \ker A_v,
$$
the   inclusion for the kernel follows. To see the equality assertion, 
if $H^0(E^{*}\otimes L)\to H^0(E^{*}\otimes L_{|Z})$
is surjective, for every $j=1\hd  k$ we can choose $t_{h,j}\in H^0(E^*\otimes L)$, for $h=1\hd e$, such that $t_{h,j}(x_i)=0$ for $i\neq j$, $\forall h$ and $t_{h,j}$ span the fiber of $E^*\otimes L$ at $x_j$.
It follows that if $s\in \ker(A_v)$ then {$t_{h,j}(x_j)\cdot s(x_j)=0$} for every $h,j$
which implies $s(x_j)=0$, that is $s\in H^0(I_Z\otimes E)$. 
The dual statement is similar.
\end{proof}

The following theorem gives a useful   criterion
to find local equations of secant varieties.  

\begin{theorem}\label{irredcrit}

Let $v=\sum_{i=1}^r x_i \in V$ and let $Z=\{[x_1],\ldots , [x_r]\}$, where $[x_j]\in X$.
If 
$$H^0(I_{Z}\otimes E)\otimes H^0(I_{Z}\otimes E^{*}\otimes L) \rig{\ } H^0(I_{Z^2}\otimes L)
$$
is surjective,  then  $\sigma_r(X)$ 
is an  irreducible component of  $Rank_r(E)$.

If $(E,L)$ is a symmetric, resp. skew-symmetric  pair,  
and 
 \begin{align*} S^2\left(H^0(I_{Z}\otimes E)\right)&\to H^0(I_{Z^2}\otimes L)\\
 {\rm resp.\ \ }
 \wedge^2\left(H^0(I_{Z}\otimes E)           \right)&\to H^0(I_{Z^2}\otimes L)
 \end{align*}
is surjective,  then  $\sigma_r(X)$ 
is an  irreducible component of  $Rank_r(E)$. 
\end{theorem}
\begin{proof}
Write $v=x_1+\cdots + x_r$ for a smooth point of $\hat\s_r(X)$.
Recall that by Terracini's Lemma, 
$\BP\hat N^*_{[v]}\s_r(X)$ is the space of hyperplanes $H\in \BP V^*$
such that  $H\cap  \s_r(X)$ is singular at the 
$[x_i]$, i.e., $\hat N^*_{[v]}\s_r(X)=H^0(I_{Z^2}\otimes L)$.
Consider the commutative diagram
$$
\begin{array}{ccc}
H^0(I_{Z}\otimes E)\otimes H^0(I_{Z}\otimes E^{*}\otimes L)& \rig{\ }& H^0(I_{Z^2}\otimes L)\\
\dow{i}&&\dow{i} \\
\ker A_{v}\otimes \textrm{Im\ }A_{v}^{\perp}&\rig{\ }& H^0(L)\end{array}
$$

The surjectivity of the map in the first row
implies that the rank of map in the second row is at least  $\dim H^0(I_{Z^2}\otimes L)$.
By Proposition \ref{kerdec},  $\ker A_v\otimes \textrm{Im\ }A_{v}^{\perp}\to H^0(I_{Z^2}\otimes L)$
is surjective, so that the conormal spaces of $\sigma_k(X)$ and 
of $Rank_k(E)$ coincide at $v$, proving  the general case.
The symmetric and skew-symmetric cases are analogous. 
\end{proof}

\section{The induction Lemma}\label{indthmsect}
 
\subsection{Weak defectivity}

  The notion of weak defectivity, which dates back to Terracini,
was applied by Ciliberto and Chiantini in  \cite{MR1859030} to show
many cases of tensors  and symmetric tensors admitted a  unique 
decomposition as a sum of rank one tensors. We review here it as it is
used to prove the promised induction Lemma \ref{zeta}.
 
\begin{definition} \cite[Def. 1.2]{MR1859030}
A projective variety $X\subset \BP V$ is {\it  $k$-weakly defective}  if the general hyperplane
tangent in $k$ general points of $X$ exists and   is tangent along a variety of positive dimension.
\end{definition}

\noindent{\bf Notational warning}: what we call
$k$-weakly defective is  called $(k-1)$-weakly defective in \cite{MR1859030}. We shifted the index
in order to uniformize to our notion of $k$-defectivity: by Terracini's lemma, 
  $k$-defective varieties (i.e., those where $\tdim\s_k(X)$ is less than the expected dimension) are also $k$-weakly defective.

The Veronese varieties which are weakly defective have been classified.

\begin{theorem}[Chiantini-Ciliberto-Mella-Ballico]\label{weakclas}\cite{MR1859030,MR2238925,ballico}
The $k$-weakly defective varieties $v_d(\P n)$  are  the triples $(k,d,n)$:

(i) the $k$-defective varieties, namely $(k,2,n)$, $k=2,\ldots,{{n+2}\choose 2}$, 
$(5,4,2)$, $(9,4,3)$, $(14,4,4)$, $(7,3,4) $,

and 

(ii) $(9,6,2)$, $(8,4,3)$.  
\end{theorem}

Let $L$ be an ample line bundle on a variety $X$. Recall that the {\it discriminant variety} of a 
subspace $V\subseteq H^0(L)$ is given by the elements of $\BP  V $ whose zero sets (as sections of $L$) are singular
outside the base locus of common zeros of elements of $V$.
When $V$ gives an embedding of $X$, then the discriminant variety coincides with the dual variety
of $X\subset \BP V^*$.

\begin{proposition}\label{nonlinear}
Assume that $X\subset\BP V= \BP H^0(L)^*$ is not $k$-weakly defective and that $\sigma_k(X)$
is a proper subvariety of $\BP V$. Then
 for a general $Z'$ of length $k'< k$, the discriminant subvariety in
$\BP H^0(I_{Z'^2}\otimes L)$, given by the hyperplane sections having an additional singular point outside $Z'$,
is not contained in any hyperplane.
\end{proposition}

\begin{proof} If $X$ is not  $k$-weakly defective, then it is  not $k'$-weakly defective
for any $k'\le k$. Hence we may assume $k'=k-1$. Consider the   projection $\pi$ of $X$ centered at the span of $k'$ general tangent spaces 
at $X$.
The image $\pi(X)$ is not $1$-weakly defective,  \cite[Prop 3.6]{MR1859030}   which means that the Gauss map of $\pi(X)$ is nondegenerate
\cite[Rem. 3.1 (ii)]{MR1859030}.
Consider the dual variety of $\pi(X)$, which is contained in the discriminant, 
 see \cite[p. 810]{lanterimunoz}.
If the dual variety of $\pi(X)$ were contained in a hyperplane, then   $\pi(X)$ would be
 a cone (see, e.g. \cite[Prop. 1.1]{E2}), and thus be $1$-weakly defective, which is a contradiction.
Hence the linear span of the discriminant  variety is the ambient space.
\end{proof}

\subsection{If $X$ is not  $k$-weakly  defective and the criterion
of \S\ref{critsect} works for $k$, it works   for $k'\le k$}

\begin{lemma}\label{zeta}Let $X\subset\BP V= \BP H^0(L)^*$ be a variety and $E\ra X$ be a vector bundle on $X$. 

Assume that 

$$H^0(I_{Z}\otimes E)\otimes H^0(I_{Z}\otimes E^{*}\otimes L) \rig{\ } H^0(I_{Z^2}\otimes L)$$
is surjective for the general $Z$ of length $k$,  
that 
$X$ is not $k$-weakly defective,  and that $\sigma_k(X)$
is a proper subvariety of $\BP V$.

Then 
$$H^0(I_{Z'}\otimes E)\otimes H^0(I_{Z'}\otimes E^{*}\otimes L) \rig{\ } H^0(I_{Z'^2}\otimes L)
$$ 
is surjective for  general $Z'$ of length $k'\le k$.

The analogous  results hold  in the symmetric and  skew-symmetric cases.

\end{lemma}

\begin{proof} It is enough to prove the case when $k'=k-1$. Let $s\in H^0(I_{Z'}^2\otimes L)$.
By the assumption and Proposition \ref{nonlinear} there are sections  $s_1, \ldots, s_t$, with $t$ at most $\tdim H^0(I_{Z'}^2\otimes L)$,  with   singular points
respectively   $p_1, \ldots, p_t$ outside $Z'$, such that $s=\sum_{i=1}^ts_i$. Let $Z_i=Z'\cup \{p_i\}$ for
$i=1,\ldots , t$. We may assume that the $p_i$ are in general linear position.
By   assumption $s_i$ is in the image of 
$H^0(I_{Z_i}\otimes E)\otimes H^0(I_{Z_i}\otimes E^{*}\otimes L) \rig{\ } H^0(I_{Z_i^2}\otimes L)$.
So all $s_i$ come from $H^0(I_{Z'}\otimes E)\otimes H^0(I_{Z'}\otimes E^{*}\otimes L)$.  The symmetric and skew-symmetric cases are analogous.\end{proof}

\begin{example}[Examples where downward induction fails]  
  Theorem \ref{weakclas} (ii)  furnishes  cases where the hypotheses of 
Lemma \ref{zeta} are not satisfied. Let $k=9$ and $L=\O_{\P 2}(6)$.
Here $9$ general singular points impose independent conditions on sextics,
but for $k'=8$, all the sextics singular at eight general points
  have an additional singular point, namely  the ninth point given by the intersection
of all cubics through the eight points, so it is not  general.
Indeed,  in this case $\sigma_9(v_6(\P 2))$ is the catalecticant hypersurface
of degree $10$ given by the determinant of $\phi_{3,3}\in S^3\BC^3\ot S^3\BC^3$, but the $9\times 9$ minors of $\phi_{3,3}$
define a variety of dimension $24$ which strictly contains $\sigma_8(v_6(\P 2))$, 
which has dimension $23$.

Similarly,  the $9\times 9$ minors of $\phi_{2,2}\in S^2\BC^4\ot S^2\BC^4$ define $\sigma_8(v_4(\P 3))$,
but the $8\times 8$ minors of $\phi_{2,2}$
define a variety of dimension $28$ which strictly contains $\sigma_7(v_4(\P 3))$,  
which has dimension $27$.
\end{example}

\section{Proof of Theorem \ref{mainyoung}}\label{mainypfsect}
Recall $a=\lfloor \frac n2\rfloor$, $d=2\d+1$, $V=\BC^{n+1}$.  By Lemma \ref{zeta},   it is sufficient to prove the case
$t=\binom{\d+n}n$.  
Here $E= \wedge^{n-a}Q(\d))$ where $Q\ra \BP V$ is the tautological quotient bundle.






By   Theorem \ref{irredcrit}
it is sufficient  to prove the map
\be\label{wz}
H^0(I_Z\otimes \wedge^aQ(\d ))\otimes H^0(I_Z\otimes \wedge^{n-a}Q(\d))\to H^0(I_Z^2(2\d+1))
\ene
is surjective.
In the case $n=2a$, $a$ odd we have to prove that
the map
\be\label{wzsk}  \wedge^2H^0(I_Z\otimes \wedge^aQ(\d))\to H^0(I_Z^2(2\d+1))
\ene
is surjective. The arguments are similar.

We need the following lemma, whose proof is given below:

\begin{lemma} \label{starodd}
Let $Z$ be a set of ${{{\d}+n }\choose {n}}$ points in $\P n$ obtained as the
intersection of ${\d}+n $ general hyperplanes $H_1\hd  H_{\d+n}$ (${\d}\ge 1$).

Let $h_i\in V^*$ be an equation for $H_i$. A basis of the  space of polynomials of degree $2{\d}+1$ which are singular on $Z$ is given by
$P_{I,J}:=\left(\prod_{i\in I}h_i\right)\cdot\left(\prod_{j\in J}h_j\right)$ where  
$I, J\subseteq\{1,\ldots {\d}+n \}$
are multi-indices (without repetitions)  satisfying  $|I|={\d+1}$, $|J|={\d} $,  and $|I\cap J|\le {\d}-1$ (that is $J$ cannot be contained in $I$).
\end{lemma}

We prove that the map \eqref{wz} is surjective for
$t={{\d+n}\choose n}$ by degeneration to the case when
the ${{\d+n}\choose n}$ points are the vertices of a configuration given by the
union of $\d+n$ general hyperplanes given by linear forms $h_1,\ldots , h_{\d+n}\in V^*$.

First consider  the case $n$ is even.
For any multi-index $I\subset\{1,\ldots , \d+n\}$ with $|I|=\d$
write $q_I=\prod_{k\in I}h_k\in H^0(\O(\d))$  and for any 
multi-index $H$, of length $a+1$, let $s_{H}$ be the section
of $\wedge^aQ$ represented by the linear subspace of dimension $a-1$
given by $h_H=\{h_i=0, i\in H\}$.
Indeed the linear subspace is a decomposable element of $\wedge^aV=H^0(\wedge^aQ)$.
The section $s_H$ is represented by the Pl\"ucker coordinates of $h_H$.

The section $q_{I}s_H\in H^0(\wedge^aQ(\d))$ vanishes on the reducible variety consisting of the linear subspace 
 $h_{H}$ of dimension $a-1$
and the degree $\d$ hypersurface $q_{I}=0$.
In particular, if $I\cap H=\emptyset$, then $q_Is_{H}\in H^0(I_Z\otimes \wedge^aQ(\d))$.

Let $I=I_0\cup\{u\}$ with $|I_0|=\d$, consider $|J|=\d$ such that
$u\notin J$ and $|I\cap J|\le \d-1$.
It is possible to choose  $K_1, K_2$ such that $|K_1|=|K_2|=a$ and such that $K_1\cap K_2=K_1\cap I=K_2\cap J=\emptyset$.

We get $q_{I_0}s_{\{u\}\cup K_1}\wedge q_{J}s_{\{u\}\cup K_2}$ is the degree $2\d+1$ hypersurface
 $q_{I_0}q_{J}h_u=q_Iq_J$. By Lemma \ref{starodd} these hypersurfaces generate $H^0(I_Z^2(2\d+1))$.
Hence $W_Z$ is surjective and the result is proved in the case $n$ is even.

\smallskip 

When  $n$  is odd, 
the morphism $A_\phi$ is represented by a rectangular matrix,
and it induces a map
$B\colon H^0(\wedge^aQ(\d))\otimes H^0(\wedge^{n-a}Q(\d))\to H^0(\O(2\d+1))$.

For any
multi-index $H$, of length $n-a+1$, let $s_{H}$ be the section
of $\wedge^aQ$ represented by the linear subspace of dimension $a-1$
given by $h_H=\{h_i=0, i\in H\}$.
Indeed the linear subspace is a decomposable element of $\wedge^aV=H^0(\wedge^aQ)$.
The section $s_H$ is represented by the Pl\"ucker coordinates of $h_H$.

The section $q_{I}s_H\in H^0(\wedge^aQ(\d))$ vanishes on the reducible variety consisting of the linear subspace 
 $h_{H}$ of dimension $a-1$
and the degree $\d$ hypersurface $q_{I}$.
In particular, if $I\cap H=\emptyset$, then $q_Is_{H}\in H^0(I_Z\otimes \wedge^aQ(\d))$.

Let $I=I_0\cup\{u\}$ with $|I_0|=\d$, consider $|J|=\d$ such that
$u\notin J$ and $|I\cap J|\le \d-1$.
The modification to the above proof is that it is possible to choose  $K_1, K_2$ 
such that $|K_1|=n-a=a+1$, $|K_2|=a$ and such that $K_1\cap K_2=K_1\cap I=K_2\cap J=\emptyset$.

Then  $q_{I_0}s_{\{u\}\cup K_1}\in H^0(\wedge^{a+1}Q(\d))$,
$q_{J}s_{\{u\}\cup K_2}\in H^0(\wedge^{a}Q(\d))$, 

and $q_{I_0}s_{\{u\}\cup K_1}\wedge q_{J}s_{\{u\}\cup K_2}$ is the degree $2\d+1$ hypersurface
 $q_{I_0}q_{J}h_u=q_Iq_J$. By Lemma \ref{starodd} these hypersurfaces generate $H^0(I_Z^2(2\d+1))$.
Hence the map \eqref{wz} is surjective and the result is proved. \qedd

\subsection*{Proof of Lemma \ref{starodd}} 
Let  $Z$ be a set of ${{{\d}+n }\choose {n}}$ points in $\P n=\BP V$ obtained as the
intersection of ${\d}+n $ general hyperplanes $H_1,\ldots , H_{{\d}+n}$ (${\d}\ge 1$), and
$h_i\in V^*$ is an equation for $H_i$.

We begin by discussing    properties of
the hypersurfaces through $Z$, which may be  of  independent interest.

\begin{proposition}\label{basism}\

(i) For $d\ge {\d}$, the products $\prod_{i\in I}h_i$ for every $I\subseteq\{1,\ldots, n+{\d} \}$ such that $|I|=d$ are independent in $S^dV^*$.

(ii)  For $d\le {\d}$, the products $\prod_{i\in I}h_i$ for every $I\subseteq\{1,\ldots, n+{\d} \}$ such that $|I|=d$ span 
$S^dV^*$.

\end{proposition}

\begin{proof}  Write $h_I=\prod_{\in I}h_i$.
Consider a linear combination $\sum_Ia_Ih_I=0$. Let $d\ge {\d}$. If $d\geq n+\d$ the statement is vacuous, so
assume $d<n+\d$.  For each  $I_0$ such that $|I_0|=d$ 
 there is a point $P_0\in V$ such that $h_i(P_0)=0$ for $i\notin I_0$
and $h_i(P_0)\neq 0$ for $i\in I_0$. The equality $\sum_Ia_Ih_I(P_0)=0$ implies  $a_{I_0}=0$, 
proving (i). For $d={\d}$,  (ii) follows as well by counting dimensions, as $\tdim S^{\d}V^*={{{\d}+n}\choose n}$.
For $d\le {\d}-1$ pick $f\in S^dV^*$.
  By the case already proved, there exist coefficients $a_I$, $b_J$ such that
\begin{equation}\label{fgen}
f\prod_{i=1}^{{\d}-d}h_i=\sum_Ia_Ih_I+\sum_Jb_Jh_J
\end{equation}
 where the first sum is over all $I$ with $|I|={\d}$ such that
$\{1,\ldots , {\d}-d\}\subseteq I$ and the second sum
over all $J$ with $|J|={\d}$ such that
$\{1,\ldots , {\d}-d\}\not\subseteq J$.

For every $J_0$ such that $|J_0|={\d}$ and $\{1,\ldots , {\d}-d\}\not\subseteq J_0$
there is   a (unique) point $[P_0]$  such that $h_i(P_0)=0$ for $i\notin J_0$
and $h_i(P_0)\neq 0$ for $i\in J_0$. In particular $\prod_{i=1}^{{\d}-d}h_i(P_0)=0$ and
substituting $P_0$ into \eqref{fgen} shows 
  $b_{J_0}=0$,  for each $J_0$. Thus
$$f\prod_{i=1}^{{\d}-d}h_i=\sum_Ia_Ih_I$$
and   dividing both sides by $\prod_{i=1}^{{\d}-d}h_i$  proves (ii). 
\end{proof}

Recall that the subspace of $S^dV^*$ of  polynomials which pass through
$p$ points and are singular through $q$ points has codimension $\le \min\left({{n+d}\choose n},p+q(n+1)\right)$.
When   equality holds one says that the conditions imposed by the points are independent
and that the subspace  has the expected codimension.

Let $I_Z$ denote  the ideal sheaf of $Z$ and $I_Z(d)=I_Z\otimes\O_{\P n}(d)$. These sheaves have cohomology groups $H^p(I_Z(d))$, where in particular
$H^0(I_Z(d))=\{ P\in S^dV^*\mid Z\subseteq \tzeros(P)\}$.

\begin{proposition}\label{castmumford} Let $Z$ be a set of $\binom{n+\d}n$ points in $\pp n$ obtained as above. Then

(i) $H^0(I_Z(d))=0$ if $d\le {\d}$.

(ii) $H^0(I_Z({\d+1}))$ has dimension ${{{\d}+n}\choose {\d+1}}$ and it is generated by the products $h_I$ with $|I|= {\d}+1$.

(iii) If $d\ge {\d}+1$ then $H^0(I_Z(d))$ is generated in degree ${\d}+1$, that is
the natural morphism $H^0(I_Z({\d}+1))\otimes H^0(\O_{\P n}(d-{\d}-1))\rig{}H^0(I_Z(d))$
is surjective. 
\end{proposition}

\begin{proof} Consider the exact sequence of sheaves
$$0\rig{}I_Z(d)\rig{}\O_{\P n}(d)\rig{}\O_Z(d)\rig{}0.
$$
Since $Z$ is finite,   $\tdim H^0(\O_Z(d))= {{{\d}+n}\choose n}=\deg Z$
for every $d$.

 The space  $H^0(\O_{\P n}({\d}))$ has dimension ${{{\d}+n}\choose n}$
and by  Proposition  \ref{basism}, for every $I$ with $|I|={\d}$ the element $h_I$ vanishes on all the points of $Z$ with just one exception,
given by $\cap_{j\notin I}H_j$. Hence the restriction map $H^0(\O_{\P n}({\d}))\rig{}H^0(\O_Z({\d}))$
is an isomorphism. It follows that $H^0(I_Z({ \d}))=H^1(I_Z({ \d}))=0$, and thus $H^0(I_Z({  d}))=0$
for $ d<\d$ proving (i).

Since $\dim Z=0$,   $H^i(\O_Z(k))=0$ for $i\ge 1$, and all  $k$.
>From this it follows that $H^i(I_Z(\d-i+1))=0$ for $i\ge 2$,
because $H^i(\O_{\BP^n}(\d-i+1))=0$. 
  The vanishing for $i=1$ was   proved above,
so that $I_Z$ is $({\d}+1)$-regular and by the Castelnuovo-Mumford criterion \cite[Chap. 14]{MR0209285}
$I_Z({\d}+1)$ is globally generated, $H^1(I_Z(k))=0$ for $k\ge {\d}+1$, and   part (iii) follows.

In order to prove (ii), consider the products $h_I$ with $|I|= {\d}+1$,
which are independent by   Proposition \ref{basism}.i, so   they span  a   ${{{\d}+n}\choose {\d+1}}$-dimensional
subspace  of $H^0(I_Z({\d}+1))$.

The long exact sequence in  cohomology implies  
\begin{align*}h^0(I_Z({\d}+1))&=h^0(\O_{\P n}({\d}+1))-h^0(\O_Z({\d}+1))+h^1(I_Z({\d}+1))\\
&=
{{{\d}+1+n}\choose n}-{{{\d}+n}\choose n}+0={{{\d}+n}\choose {{\d}+1}}
\end{align*}
which concludes the proof. 
\end{proof}

\begin{proposition}\label{sing1} Notations as above.
Let $L_0\subset \{ 1\hd n+\d\}$ have cardinality $n$.

The space of polynomials of degree ${\d}+1$ which pass through the points $y_L$ and are singular at $y_{L_0}$
has a basis given by the products $\prod_{i\in J}h_i$ for every $J\subseteq\{1,\ldots, n+{\d} \}$ such that $|J|={\d}+1$ and
$\#\left(J\cap L_0\right)\ge 2$. This space has the expected codimension
$\binom{\d+n}n +n$. 
\end{proposition}

\begin{proof}  Note that if $|J|={\d}+1$ then
$\#\left(J\cap L_0\right)\ge 1$ and   equality holds just for the $n$ products $h_ih_{L_0^c}$
for $i\in L_0$, where $L_0^c=\{ 1\hd n+\d\}\backslash L_0$. Every linear combination of the $h_J$ which has a nonzero coefficient
in these $n$ products is nonsingular at $[y_{L_0}]$. 

Hence the products which are different from these $n$ generate the space
of polynomials of degree ${\d}+1$ which pass through the points $[y_L]$ and are singular at $[y_{L_0}]$.

There are  ${{{\d}+n}\choose {n-1}}-n$ such polynomials,
which is  the expected number  ${{{\d}+n+1}\choose n}-(n+1)-\left[{{{\d}+n}\choose n}-1\right]$. 
Since the codimension is always at most  the expected one, it follows
that these generators give a basis.
\end{proof}

In the remainder of this section, we use  products $h_I$ where $I$ is a   multi-index where repetitions are allowed.
Given such a multi-index  $I=\{i_1\hd i_{|I|}\}$, we write $h_I=h_1^{k_1}\cdots h_{n+\d}^{k_{n+\d}}$, where
$j$ appears $k_j$ times in $I$ and we write  $|I|=k_1+\cdots + k_{n+\d}$.
The {\it support } $s(I)$ of a multi-index  $I$ is the set of $j\subset \{ 1\hd n+\d\}$ with $k_j>0$.

An immediate consequence of (ii) and (iii) of   Proposition  \ref{castmumford} is:
\begin{corollary}\label{4012}
For $d\ge {\d}+1$, the vector space $H^0(I_Z(d))$ is generated by the monomials $h_I$
with $|I|=d$ and  $|s(I)|\ge {\d}+1$.
\end{corollary}

\begin{proposition}\label{exp2}The space of polynomials of degree $2{\d}+1$ which contain $Z$
is generated by products $h_I$ with $|I|=2\d+1$, $|s(I)|\ge {\d}+1$, 
and all the exponents in $h_I$ are at most $  2$.
\end{proposition}

\begin{proof} By  Corollary \ref{4012}, it just remains  to prove   the statement about the exponents.
Let $h_I=\prod_{i=1}^{{\d}+n}h_i^{k_i}$ and let
$n_j(I)=\#\{i|k_i=j\}$ for $j=0,\ldots, 2{\d}+1$. Hence $\sum_{j\ge 1}jn_j(I)=2{\d}+1$
and $\sum_{j\ge 1}n_j(I)\ge {\d}+1$. 

Assume that $\gamma:=\sum_{j\ge 3}n_j(I)> 0$. It is enough to show that $h_I=\sum c_Jh_J$
where every   $J$ which appears in the sum satisfies $|s(J)|\ge {\d}+1$, $|J|=2{\d}+1$ and  $\sum_{j\ge 3}n_j(J)<\gamma$.

Indeed 
$$2{\d}+1=n_1(I)+2n_2(I)+3n_3(I)+\ldots\ge n_1(I)+2n_2(I)+3\gamma\ge \left({\d}+1-n_2(I)-\gamma\right) 
+2n_2(I)+3\gamma
$$
that is, 
$n_2(I)+\gamma\le {\d}-\gamma\le {\d}-1$. Hence there at least $n+1$ forms $h_i$ which appear with exponent at most one in $h_I$.
 We can express all the remaining forms $h_s$ as linear combinations of these $n+1$ forms.
By expressing a form with exponent at least $3$ as linear combination of these $n+1$ linear forms $h_i$,
we get $h_I=\sum c_Jh_J$
where each summand has the required properties. 
\end{proof}

To better understand the numbers in the following proposition, recall that
$\sum_{i=1}^{\d+1}{{{\d}+n -i}\choose {n-1}}={{{\d}+n }\choose n}$ as taking $E=\BC^n$, $F=\BC^1$, one has
$S^n(E+F)=S^nE\op S^{n-1}E\ot F\op S^{n-2}E\ot S^2F\op\cdots\op S^nF$.

\begin{proposition}\label{sing2}
For $k=0,\ldots, {\d} $, let $I_{{\d},k,n}$ be the linear system of hypersurfaces of degree $2{\d}+1-k$ which 
contain $Z$
and are singular
on the points of $Z$ which lie outside $\cup_{i=1}^{k}H_i$.   
The space $I_{{\d},k,n}$ has the expected codimension 
$(n+1){{{\d}+n }\choose n}-n\sum_{i=1}^k{{{\d}+n -i}\choose {n-1}}$ .
\end{proposition}

\begin{proof}
For $k={\d}$ the assertion  is Proposition  \ref{sing1} with $L_0=\{ \d+1\hd \d+n\}$. We work by induction on $n$ and by descending induction on $k$
(for fixed ${\d}$).
We restrict to the last hyperplane $H_{{\d}+n}$.

Let $V_{{\d},k+1,n}\supseteq I_{{\d},k+1,n}$ denote the set of    hypersurfaces
of degree $2{\d}-k$ which pass through $Z\setminus \{H_{{\d}+n}\cap\left(\cup_{i=1}^{k}H_i\right)\}$
and are singular on the points of $Z$ which lie outside $\left(\cup_{i=1}^{k}H_i\right)\cup \{H_{{\d}+n}\}$.
We have the exact sequence
$$0\rig{}V_{{\d},k+1,n}\rig{\psi}I_{{\d},k,n}\rig{\phi}I_{{\d},k,n-1}
$$
where $\phi$ is the restriction to $\BC^n\subset \BC^{n+1}$, and $\psi$ is the multiplication by $h_{{\d}+n}$.

Since by induction $I_{{\d},k+1,n}$ has the expected codimension in $S^{2{\d}-k}V^*$  
it follows that also $V_{{\d},k+1,n}$ has the expected codimension 
$$n\left[{{{\d}+n-1}\choose n}-\sum_{i=1}^k{{{\d}+n-1-i}\choose {n-1}}\right]+
{{{\d}+n}\choose n}-\sum_{i=1}^k{{{\d}+n-1-i}\choose {n-2}}$$
because the conditions imposed by $V_{{\d},k+1,n}$ are a subset of the conditions imposed by
$I_{{\d},k+1,n}$ and a subset of a set of independent conditions still consists of independent conditions.

Since the codimension of $I_{{\d},k,n-1}$ in $S^{2{\d}+1-k}\BC^{n}$ is the expected one,
it follows that the codimension of $I_{{\d},k,n}$ in $S^{2{\d}+1-k}V^*$ is at least    the expected one,
hence   equality holds. 
\end{proof}
 
The following Corollary proves   Lemma \ref{starodd}.

\begin{corollary}\label{hypsingular}
Let $Z$ be a set of ${{{\d}+n }\choose {n}}$ points in $\P n=\BP V$ obtained as the
intersection of ${\d}+n $ general hyperplanes $H_1\hd H_{{\d}+n}$ (${\d}\ge 1$), and let
$h_i\in V^*$ be an equation of $H_i$.
The linear system of hypersurfaces of degree $2{\d}+1$ which are singular on $Z$
has the expected codimension $(n+1){{{\d}+n}\choose n}$ and it is generated by the products  $h_I=\prod_{i\in I}h_i$ with 
$I$ a multi-index allowing repetitions   such that $|I|=2{\d}+1$ and the support of $I$ is at least $ {\d}+2$
and all the exponents are at most two. 
\end{corollary}
\begin{proof} The statement about the codimension is the case $k=0$ of   Proposition \ref{sing2}. 
Note that every product $h_I$ with 
$I$   such that $|I|=2{\d}+1$ and $s(I)\ge {\d}+2$    is singular at any $P\in Z$ because $h_I$ contains at least two factors which vanish
at $P$. Let $f$ be a homogeneous polynomial of degree $2{\d}+1$ which is singular on $Z$.
By   Proposition \ref{exp2} we have the decomposition $f=\sum_{|s(I)|\ge {\d}+2}a_Ih_I+\sum_{|s(I)|= {\d+1}}b_Ih_I$,
where all the exponents are at most $2$. Let $I_0$ be     in the second summation with $|s(I_0)|= {\d}+1$.
There is a unique $h_i$ appearing in $h_{I_0}$ with exponent $1$ and $n-1$ hyperplanes appearing with exponent $0$.
Let $P_0$ be the point where these $1+(n-1)=n$ hyperplanes meet.
It follows that $h_{I_0}$ is nonsingular at $P_0$ while, for all the other summands, $h_I$ is singular in $P_0$.
It follows that $b_{I_0}=0$ and $f=\sum_{|s(I)|\ge {\d}+2}a_Ih_I$ as we wanted. 
\end{proof}

\section{Construction of $A_v^E$ via a presentation of $E$}\label{factoringsect}

\subsection{Motivating example}\label{motexsect}
Set $\tdim V=n+1$,  $d=2\d+1$ and $a=\lfloor \frac n2\rfloor$. We saw in \S\ref{yflatsect} that the
inclusion 
$S^dV\subset S_{\d+1,1^{n-a}}V^*\ot  S_{\d+1,1^a}V$ yields  new modules of equations for
$\s_r(v_d(\BP V))$. The relevant vector bundle here is 
$E=E_{ \d \o_1+ \o_{n-a+1}}= {\wedge^a Q(\d)}$  {where $Q$ is the tautological  quotient bundle}. The equations were
 initially presented  
  without reference to representation theory using \eqref{YFmap}.
Note that $\trank(YF_{d,n}(w^d))=\binom{n}{a}=\trank(E)$.

This   map between larger, but more elementary,   spaces has the same rank properties as the map $S_{\d+1,1^{n-a}}V^*\ra S_{\d+1,1^a}V$, because
when one decomposes the spaces in this map as $SL(V)$-modules, the other module maps are zero.

Observe that 
\begin{align*}
S^{\d}V^*\ot \wedge^a V&=H^0(\cO(\d)_{\BP V}\ot \wedge^aV),
\\
{\rm and\ } (S^{\d}V\ot \wedge^{a+1}V)^*&= H^0(\cO(\d)_{\BP V}\ot \wedge^{n-a+1}V).
\end{align*}

The vector bundle   $E {=\wedge^a Q(\d)}$ may  be recovered as the image of the map
 $$L_1=\O(\d)\otimes \wedge^a{V}\rig{p_E}\O(\d+1)\otimes \wedge^{a+1}{V}=L_0.
$$

\subsection{The general case}
In the general set-up,
in order to   factorize $E$, one begins with a base line bundle $M\ra X$, which
is $\cO_{\BP W}(1)$ when $X=v_d(\BP W)\subset \BP S^dW=\BP V$, and more generally is the ample generator of the Picard group
if $X$ has Picard number one (e.g., rational homogeneous varieties $G/P$ with $P$ maximal).

The key is to realize $E$ as the image of a map
$$
p_E : L_1\ra L_0
$$
where the
$L_i$ are direct sums of powers of $M$
and 
$$
\tim(p_E)=E.
$$

By construction,  
for   $x\in \hat X$,  $\trank(p_E(x))= e$.
Let   $v\in V=H^0(X,L)^{*}$.
Compose the map $H^0(L_0)\otimes H^0(L)^{*}\to H^0(L_0^{*}\otimes L)^{*}$
   with the map on sections induced from $p_E$,  to obtain  

$$P_v^E\colon H^0(L_1)\to H^0(L_0^{*}\otimes L)^{*}$$

and observe the linearity $P_{sv_1+tv_2}^E=sP_{v_1}^E+tP_{v_2}^E$ for $v_1,v_2\in V$, $s,t\in \BC$.

\subsection{Case $X=v_d(\BP W)$}

In the case $(X,M)=(\BP W,\cO(1))$, and $L=\cO(d)$, this means
$$L_j=\oplus (\cO(i))^{\op a^i_j}
$$
for some finite set of natural numbers $a^i_1,a^i_2$.
For the examples in this paper there is just one term in the summand, 
e.g., in \S\ref{motexsect}, 
$L_1=\cO(\d)^{\op \binom{n+1}a}$, $L_0=\cO(\d+1)^{\op \binom{n+1}{a+1}}$.

When $X=v_d(\BP W)$,      $P_v^E$ may be written explicitly in bases as follows: 

Let $L_1=\oplus_{j=1}^{m_1}\O(b_j)$ and $L_0=\oplus_{i=1}^{m_0}\O(a_i)$. Let  $p_{ij}\in S^{a_i-b_j}W^*$ denote
the   map $\cO(b_j)\ra \cO(a_i)$, given by
the composition
$$\cO(b_j)\ra L_1\rig{p_E}L_0\ra\cO(a_i)$$

Then, for any $\phi\in S^dW$, 
$P_v^E(\phi)$ is obtained by taking a 
matrix of $m_1\times m_0$ blocks with the  $(i,j)$-th block
  a matrix representing the catalecticant $\phi_{b_j,d-a_i}\in S^{b_j}W^*\otimes S^{d-a_i}W^*\otimes S^dW$ contracted by $p_{ij}$
so
the new matrix   is  of size $h^0(L_0^{*}\otimes L) \times h^0(L_1)$
 {with scalar entries}.
See Examples \ref{aronex}  and   \ref{cubic3fold} for explicit examples.

\subsection{Rank and irreducible component theorems in the presentation setting}
Under mild assumptions,
the following proposition shows that $P_v^E$ can be used in place
of $A_v^E$ in the theorems of \S\ref{constsect}.

\begin{proposition}\label{eqsigma}
Notations as above. 
Assume that
$H^0(L_1)\rig{i}H^0(E)$ and 
$H^0(L_0^{*}\otimes L)\rig{j}H^0(E^{*}\otimes L)$ are surjective.

Then the rank of $A_{v}^E$ equals the rank of $P_v^E$,
so that
 the size $(ke+1)$-minors of $P_v^E$ give equations for $\sigma_k(X)$.

If $(E,L)$ is a symmetric (resp. skew-symmetric) pair then there exists a symmetric (resp. skew-symmetric) presentation
where $L_1\simeq L_0^{*}\otimes L$ and 
  $P_v^E$ is   symmetric (resp. skew-symmetric).
\end{proposition}
\begin{proof} 
Consider the 
commutative diagram
$$\begin{array}{ccc}
H^0(L_1)&\rig{P_v}&H^0(L_0^{*}\otimes L)^{*}\\
\dow{i}&&\up{j^{t}}\\
H^0(E)&\rig{A_v}&H^0(E^{*}\otimes L)^{*}\end{array}$$
The assumptions $i$ is surjective and $j^{t}$ is injective imply
  $\trank(A_{v})=\trank(P_v)$. The symmetric and skew assertions are clear.


\end{proof}

With the assumptions of Proposition \ref{eqsigma},
the natural map
$$H^0(L_1)\otimes H^0(L_0^{*})\to H^0(E)\otimes H^0(E^{*}\otimes L)$$ is surjective
and the affine  conormal space at $v$ is  the image of
$\ker P_v\otimes  {\left(\textrm{Im\ }P_v\right)^{\perp}}\to H^0(L)$.

The following variant of Proposition \ref{construct}  is often easy to implement in practice.
 
\begin{theorem}\label{irred}
Notations as above.
Let $v\in\sigma_k(X)$. Assume  
the maps
$H^0(L_1)\rig{}H^0(E)$ and
$H^0(L_0^{*}\otimes L)\rig{}H^0(E^{*}\otimes L)$ are surjective.
\medskip

(i) If the rank of $\ker P_{v}\otimes  {\left(\textrm{Im\ }P_{v}\right)^{\perp}}\rig{g } H^0(L)$
at $v$ equals  the codimension of $\sigma_k(X)$, then
$\sigma_k(X)$ is an irreducible component of 
 $Rank_k(E)$  passing through $v$.
\medskip

(ii) If $(E,L)$ is a symmetric (resp. skew-symmetric) pair, 
  assume that the rank of $S^2\left(\ker P_{v}\right)\to H^0(L)$
(resp. of $\wedge^2\left(\ker P_{v}\right)\to H^0(L)$)
at $v$ coincides with the codimension of $\sigma_k(X)$.

Then  
$\sigma_k(X)$ is an irreducible component of 
 $Rank_k(E)$  passing through $v$.
\end{theorem}

Note that in the skew-symmetric case, $Rank_k(E)$ is defined by sub-Pfaffians.

\begin{remark} In Theorem \ref{irred},  the rank of the maps is always bounded above by
  the codimension of $\sigma_k(X)$, as the rank of the maps
is equal to the dimension of the conormal space of $Rank_k(E)$ at $v$.
\end{remark}

The previous theorem can be implemented in concrete examples.
Indeed the pairing $g$, appearing in the theorem, can be described as follows: given $f\in \ker P_{v}\subset H^0(L_1)$,
 $h\in  {\left(\textrm{Im\ }P_{v}\right)^{\perp}}\subset H^0(L_0^*\ot L) $, and
$\phi\in H^0(L)^*$, one has
  $g(f,h)(\phi)=h[P_{\phi}(f)]$.

\medskip

\begin{example} \label{cubic3fold}
{\bf [Cubic 3-folds revisited]}
Let $\tdim V=5$, let  $E=\Omega^2(4)=\wedge^2T^*\BP V(4)=E_{\o_2}$ on $X=\P 4$, 
the pair $(E,\O(3))$ is a symmetric pair presented by
$$L_1=\O(1)\otimes \wedge^3 V\rig{p}\O(2)\otimes \wedge^2V=L_0.
$$
Choose a basis $x_0\hd x_4$ of $V$, so 
  $p$ is represented by the $10\times 10$ symmetric matrix
called  $K_4$ in the introduction. 
Let $L=\O(3)$,  
for any $\phi\in S^{3}V^{*}$ the map
$A_\phi^E$   is the skew-symmetric morphism from
$H^0(\Omega^2(4))$ to its dual $H^0(\Omega^3(5))^{*}$
where both have  dimension $45$
and $P_\phi$ is represented by the $50\times 50$ block matrix
where the entries $\pm x_i$ in $K_4$ are replaced with $\pm (\frac{\partial\phi}{\partial x_j})_{1,1}$, which are
$5\times 5$ symmetric catalecticant matrices of the quadric $\frac{\partial\phi}{\partial x_j}$.
In \cite{ottwaring} it is shown that a matrix representing $A_\phi^E$
is obtained from  the matrix representing  $P_\phi$ by deleting five suitably chosen rows and columns.
Here $\det (A_{\phi})$ is the cube of the degree $15$ equation of $\sigma_7(v_3(\P 4))$.

Note that when $\phi=x_0^{3}$ is a cube,  then 
the rank of $A_\phi^E$ and the rank of $P_\phi$ are both equal to $6$, which is the rank of $\Omega^2(3)$.
\end{example}

\section{Decomposition of polynomials into a sum of powers}\label{decompsect}

Having a presentation for $E$ enables one to reduce the problem of decomposing a polynomial into
a sum of powers to a problem of solving a system of polynomial equations (sometimes linear) as we explain in this section.
We begin with a classical example:

\subsection{Catalecticant}
Let  $\phi\in S^dW$ be a general element of $\hat\sigma_k(v_d(\pp n))$
so it  can be written as $\phi=\sum_{i=1}^kl_i^d$. Let  $Z=\{l_1,\ldots ,l_k\}$, 
  let $\delta=\lfloor\frac{d}{2}\rfloor$, and let $k\le{{\delta+n-1}\choose n}$. Take
$E=\cO(\d)$ so 
$A_{\phi}^E\colon H^0(\O(\delta))\to H^0(\O(d-\delta))^*$. Then $\ker A_{\phi}=H^0(I_Z(\delta))$  and 
$\tker A_{\phi}\cap v_{\d}(\BP W)=\{ l_1^\d\hd l_k^\d\}$. Thus  if one can compute
the intersection, one can recover
 $Z$ and the decomposition. For a further discussion  of these ideas see \cite{OO}.

\subsection{$YF_{d,n}$}
For $\phi\in S^dW$, 
let $\delta=\lfloor\frac{d-1}{2}\rfloor$, $a=\lfloor\frac{n}{2}\rfloor$,
$k\le{{\delta+n}\choose n}$
and take $E=\wedge^aQ(\d)$, so $A_{\phi}\colon H^0(\wedge^aQ(\delta))\to H^0(\wedge^{n-a}Q(d-\delta))$,
then  $\ker A_{\phi}=H^0(I_Z\otimes \wedge^aQ(\delta))$
so that
$Z$ can be recovered as the base locus of $\ker A_{\phi}$. 

\subsection{General quintics in $S^5\BC^3$}\label{decompquintic}
A general element     $\phi\in S^5\BC^3$ is a unique sum of seven fifth powers (realized by
Hilbert, Richmond and Palatini \cite{hilblett,richm,palat}), 
write $\phi=x_1^5+\cdots +x_7^5$. Here is  an algorithm to compute
the seven summands: 

Let $E= {Q(2)}$ and consider the skew-symmetric morphism
$A_{\phi}^E\colon H^0(Q(2))\to H^0(Q(2))^{*}$.
(Note that $H^0(Q(2))=S_{3,2}V$ as an $SL_3$-module, and that   has dimension    {$15$}.)
 For a general $\phi$ the kernel of $A_{\phi}$ has dimension one and  {we let $s\in H^0(Q(2))$ denote a section spanning it}.
 {\it By   Proposition \ref{kerdec}, the seven points where $s$ vanishes correspond to the seven summands.} 
Explicitly, $s$ corresponds to  a morphism $f\colon S^2W\to W$ such that
the set $\{v\in W| f(v^2)\in <v>\}$ consists of the seven points $P_1\hd P_7$.

One way to describe the seven points is   to consider the
map 
\begin{align*}
F_{\phi}: \BP W& \ra \BP (\BC^2\ot W)
\\
 [v]&\mapsto [v\ot e_1+ f(v^2)\ot e_2] 
 \end{align*}
where $\BC^2$ has basis $e_1,e_2$. 
  Then, the seven points are the intersection
$F_{\phi}(\BP W)\cap Seg(\pp 1\times \BP W)$, which  is the set of two by two minors
of a matrix with linear and quadratic entries. Giving $W^*$ the
basis $x_0,x_1,x_2$, the matrix is of the form
$$\left[\begin{array}{ccc}
x_0&x_1&x_2\\
q_0&q_1&q_2
\end{array}\right]$$
where the $q_j$ are of degree two.
In practice this system is easily solved with, e.g. Maple.


\section{Grassmannians}\label{grassexsect}

\subsection{Skew-flattenings}
Some
 {natural}  equations for $\s_r(G(k,W))\subset \BP \wedge^k W$ are
obtained from the inclusions
$\wedge^k W\subset \wedge^a W\otimes \La{k-a}W$ which we will call
a {\it skew-flattening}. 
Let $W=\BC^{n+1}$,  let $\o_i$, $1\leq i\leq n$  denote the
fundamental weights of $\fsl(W)$, let $E_{\l}$ denote the
irreducible vector bundle induced from the $P$-module of highest weight $\l$.
The  skew-flattenings correspond to the bundle $E_{\o_a}\ra G(k,W)$.
Note that $\trank(E_{\o_a})=\binom ka$,   as $E_{\o_a}$ is isomorphic
to the $a$-th exterior power of the dual of the universal sub-bundle of rank $k$.
The map $v_{a,k-a}=A_v^E: \wedge^a W^*\ra \La{k-a}W$ thus has
rank $\binom k a$ (this is also easy to see directly).
Since $\tdim \wedge^a W=\binom {n+1}k$, one obtains 
equations possibly up to the
$\binom {n+1}k/\binom ka$-th secant variety.

The case $k=2$ is well known, $I(\s_r(G(2,W)))$ is generated in degree $r+1$
by sub-Pfaffians of size $2r+2$.

Let $E\in G(k,W)=\BG(\pp{k-1},\BP W)$. We slightly abuse notation and 
also write $E\subset W$ for the corresponding
linear subspace.
Then

\be\label{nhatgkw}
 {\hat N^*_EG(k,W)= \wedge^2 (E\upperp )\ww (\wedge^{k-2}W^*)\subset \wedge^kW^*.}
\ene
Now let $[E_1+E_2]\in \s_2(G(k,W))$.
By Terracini's lemma $\hat N^*_{[E_1+E_2]}G(k,W)= \hat N^*_{E_1}G(k,W)\cap \hat N^*_{E_2}G(k,W)$.
Let $U_{12}=E_1\upperp\cap E_2\upperp$, and $U_j= E_j\upperp$. We have

\be\label{nhats2gkw} {
\hat N^*_{[E_1+E_2]}G(k,W)=(\wedge^2 U_{12})\ww(\wedge^{k-2}W^*) 
+ U_{12}\ww U_1\ww U_2 \ww(\wedge^{k-3}W^*)
+ (\wedge^2U_1)\ww(\wedge^2 U_2)\ww (\La{k-4}W^*).}
\ene

To see this, note that each term must contain a  {$\wedge^2 (E_1\upperp)$} and a 
 {$\wedge^2 (E_2\upperp)$}. The notation
is designed so that  {$\La 2 (E_j\upperp)$} will appear if the $j$ index occurs at least twice
in the $U$'s. Similarly below, one takes all combinations of $U$'s such that each of $1,2,3$ appears
twice in the expression.

With similar notation,
 {\begin{align}\label{nhats3gkw}
N^*_{[E_1+E_2+E_3]}G(k,W)=& (\wedge^2 U_{123})\ww(\La{k-2}W^*) \\
&\nonumber
+ [U_{123}\ww U_{12}\ww U_3+  U_{123}\ww U_{13}\ww U_2+U_{123}\ww U_{23}\ww U_1]\ww
(\La{k-3}W^*)\\
\nonumber &
+U_{12}\ww U_{13}\ww U_{23}\ww (\La{k-3}W^*)
+U_{123}\ww U_{1 }\ww U_2\ww  U_3\ww(\La{k-4}W^*)
\\
\nonumber &
+ [(\wedge^2U_{12})\ww(\wedge^2 U_3)+ (\wedge^2U_{13})\ww
(\wedge^2 U_2)+(\wedge^2U_{23})\ww(\wedge^2 U_1)]  \ww 
(\La{k-4}W^*)\\
\nonumber &
+ (\wedge^2U_1)\ww(\wedge^2 U_2)\ww(\wedge^2 U_3)\ww 
(\La{k-6}W^*).
\end{align}}

 {We emphasize that with four or more subspaces analogous formulas are much more difficult,
because four subspaces have moduli, see \cite{MR0357428}.}

We now determine  to what extent the zero sets of the skew-flattenings provide
local equations for secant varieties of Grassmannians.

The first skew-flattening $\wedge^k W\subset W\ot \wedge^{k-1}W$ gives rise to the {\it subspace varieties}:
let
$$
Sub_{p}(\wedge^kW)=\{[z]\in \BP (\wedge^kW)\mid \exists W'\subset W, \ \tdim W'=p,\ z\in \wedge^k W'\}
$$
it admits a description via a Kempf-Weyman  \lq\lq collapsing\rq\rq\  of the  vector bundle
$\wedge^k\cS_p\ra G(p,  W)$, 
where $\cS_p$ is the tautological rank $p$ subspace bundle, i.e., the variety is the projectivization of the image of the total space of $\wedge^k \cS_p$ in $  \wedge^k W$. From this description one sees that
$Sub_{p}(\wedge^kW)$ is of dimension $p(n+1-p)+\binom pk$ and its affine conormal space
at a smooth point $z\in \wedge^kW'\subset \wedge^k W$ is
$$
\hat N^*_zSub_{p}(\wedge^kW)=  {\wedge^2(W'\upperp)} \ot \La{k-2}W^*
$$
As a set,  $Sub_{p}(\wedge^kW)$ is the zero locus of 
$\La{p+1}W^*\ot \La{p+1}(\La{k-1}W^*)$, the size $(p+1)$-minors of the skew-flattening.
The ideal of $Sub_{p}(\wedge^k W)$ is simply the set
of all modules $S_{\pi}W^*\subset Sym(W^*)$ such
that $\ell(\pi)>p$. Its generators are known only in very few
cases, see \cite[\S 7.3]{weyman}.

\subsection*{Question} In   what cases do the minors of the skew-flattenings
generate the ideal  of $Sub_{p}(\wedge^kW)$?

\medskip

The most relevant cases for the study of secant varieties of Grassmannians
are when $p=rk$, as
$\s_r(G(k,W))\subset Sub_{rk}(\wedge^k W)$. Note that in general $\s_r(G(k,W))$ is a much smaller subvariety
than $Sub_{rk}(\La k W)$.

\medskip

Now we compute $\tker\ot \tim\upperp$:

For $E\in G(k,W)$, write $E_{p,k-p}: \wedge^pW^*\ra \La{k-p}W$ for the
flattening.
Then
\begin{align}
\tker(E_{2,k-2})&=  {\wedge^2 (E\upperp)}\\
\tim(E_{2,k-2})\upperp &=  {(\La{k-2}E)\upperp}= E\upperp \ww  {(\La{k-3}W^*)}
\end{align}
and similarly
\begin{align}
\tker([E_1+E_2]_{2,k-2})&= U_{12}\ww W^* + U_1\ww U_2\\
\tim([E_1+E_2]_{2,k-2})\upperp &=  U_{12}\ww  {(\La{k-3}W^*)}+ U_1\ww U_2\ww  {(\La{k-4}W^*)}
\end{align}

And thus
\begin{align*}
&\tker([E_1+E_2]_{2,k-2})\ww \tim([E_1+E_2]_{2,k-2})\upperp
\\
&
=
\wedge^2U_{12}\ww  {(\La{k-2}W^*)}+ U_{12}\ww  U_1\ww U_2\ww  {(\La{k-3}W^*)}+ \wedge^2U_1\ww \wedge^2 U_2\ww  {(\La{k-4}W^*)}
\end{align*}
which agrees with \eqref{nhats2gkw}. (Note that if we had tried the $(1,k-1)$-flattening, we would
have missed the third term in \eqref{nhats2gkw}.)

Similarly
\begin{align}
\tker([E_1+E_2+E_3]_{2,k-2})&= U_{123}\ww W^* + U_{12}\ww U_3+U_{13}\ww U_2+U_{23}\ww U_1 \\
\tim([E_1+E_2+E_3]_{2,k-2})\upperp &=  U_{123}\ww  {(\La{k-3}W^*)}+ [U_{12}\ww U_3+U_{13}\ww U_2+U_{23}\ww U_1]\ww  {(\La{k-4}W^*)}\\
&\nonumber
+U_1\ww U_2\ww U_3\ww  {(\La{k-5}W^*)}
\end{align}
and $\tker\ot \tim\upperp$ does not contain   last term in \eqref{nhats3gkw}
if it is nonzero. However when  we consider
the $(3,k-3)$ flattening we will recover the full conormal space.
In general, we obtain:

\begin{theorem}\label{gkwthm}
The variety $\s_3(G(k,W))\subset \BP \wedge^kW$ is an irreducible component of the zero set
of the size $3k+1$ minors of the skew-flattening  $\wedge^sW\ot \La{k-s}W$, $s\leq k-s$, $k\leq \tdim W-k$, as long
as
$s\geq 3$.
\end{theorem}

This theorem does not extend to $\s_4(G(k,W))$ because there it is possible to have e.g.,  sums of vectors
in pairwise intersections of spaces that add to a vector in a triple intersection, which does
not occur for the third secant variety.

\begin{remark} When $\tdim W$ is small, the range of Theorem \ref{gkwthm}
can be extended. For example, when $\tdim W=8$, $\s_4(G(4,8))$ is
an irreducible component of the size $9$ minors of
$\wedge^2\BC^8\ot \wedge^6\BC^8$.
\end{remark}



We expect the situation to be similar to that of Veronese varieties, where
for small secant varieties skew flattenings provide enough equations
to cut out the variety, then for a larger range of $r$ the 
secant variety is an irreducible component of the variety
of skew flattenings, and then for larger $r$ more equations will be needed.

\subsection{Skew-inheritance} 
If $ \tdim W>m$,  and $\pi$ is a partition with
at most $m$ parts, we say the module
$S_{\pi}W$ is {\it inherited}  from the module $S_{\pi}\BC^m$.
The following is a straight-forward variant of \cite[Prop. 4.4]{LM0}:
\begin{proposition} Equations (set-theoretic, scheme-theoretic  or ideal theoretic) for
$\s_r(G(k,W))$ for $\tdim W> {kr}$ are given by
\begin{itemize}
\item the modules inherited from the ideal of $\s_r(G(k, {kr}))$, 
\item and the modules generating the ideal of
 $Sub_{kr}(\wedge^k W)$.
\end{itemize}
\end{proposition}

 \section{Homogeneous varieties $G/P$}\label{gpexamsect}

Let $X=G/P\subset \BP V_{\l}$ be a homogeneously embedded homogeneous variety, where $V_{\l}$
denotes the irreducible $G$-module of highest weight 
$\l$. In particular, when $\l=\o_i$ is fundamental, 
 $P$ is the parabolic obtained by deleting the root spaces
corresponding to the negative   roots $\a=\sum m^j\a_j $ such that $m^i\neq 0$. Let $G_0$ be
the Levi-factor of $P$, if $W_{\mu}$ is an irreducible $\fg_0$-module given
by its highest weight $\mu$, we consider it as a $\fp$-module by letting
the nilpotent part of $\fp$ act trivially. We let $E_{\mu}\ra G/P$ denote
the corresponding irreducible homogeneous vector bundle. If $\mu$ is also
$\fg$-dominant, then   $H^0(E_{\mu})^*=V_{\mu}$.

\subsection{Example: The usual flattenings for triple Segre products}
Let $A,B,C$ be vector spaces, and use $\o,\eta,\psi$ respectively for 
the fundamental weights of the corresponding $SL$-modules.
Then take $E$ with highest weight $(\o_1,\eta_1,0)$, so
$E^*\ot L$ has highest weight $(0,0,\psi_1)$.
Note each of these bundles has rank  one, as is the rank of $A_v$
when $v\in \hat  Seg(\BP A\times \BP B\times \BP C)$ as the theory predicts.
We get the usual flattenings as in \cite{MR2097214} and \cite{MR2168286}.
The equations found for the unbalanced cases in \cite[Thm. 2.4]{MR2392585}
are of this form.

\subsection{Example: Strassen's equations} Let $\tdim A=3$.
Taking 
$E$ with  highest weight $(\o_2,\eta_1,0)$ so
$E^*\ot L$ has highest weight $(\o_2,0,\psi_1)$, yields
Strassen's equations (see \cite{ottrento} and \cite{MR2387532}).
Each  of these bundles has rank two, as is the rank of $A_v$
when $v\in \hat Seg(\BP A\times \BP B\times \BP C)$.

\subsection{Inheritance from usual Grassmannians}
Let $\o$ be a fundamental weight for $\fg$, and consider $G/P \subset \BP V_{\o }$. Say the Dynkin
diagram of $\fg$ is such that $\o$ is $s$-nodes from an end node with
associated  fundamental representation $V_{\eta}$.
Then $V_{\o}\subseteq \wedge^s V_{\eta}$ and $G/P \subset G(s, V_{\eta})$,
see, e.g., \cite{LM0}. Thus 
$\s_r(G/P )\subset \s_r(G(s, V_{\eta}))$, so any equations for the latter give equations
for $\s_r(G/P)$. There may be several such ways to consider $\o$, each  will give
(usually different) modules of equations.

For example, in the case $D_{n}/P_{n-2}$,  $V_{\o_{n-2}}$ is a submodule of $ \wedge^2 V_{\o_n},  \wedge^2 V_{\o_{n-1}}$
and $\La{n-2}V_{\o_1}$. 

In the case of the $E$ series 
we have inclusions
\begin{align*}
V_{\o_3}&\subset \wedge^2 V_{\o_1} \\
V_{\o_4}&\subset \wedge^2 V_{\o_2}\\
V_{\o_{n-1}}&\subset \wedge^2 V_{\o_n}
\end{align*}
Here, in each case the vector bundle $E$ has rank two, so we obtain minimal degree
equations for the secant varieties of  $E_n/P_3,E_n/P_4,E_n/P_{n-1}$.  The
dimension of $H^0(E)$ for  $(\fe_6,\fe_7,\fe_8)$ is respectively:
($27$, $133$, $3875$) for $\o_3$, 
($78$, $912$, $147250$) for $\o_4$, and 
($27$, $56$, $248$) for $\o_{n-1}$.

\subsection{Other cases}
From the paragraph above, it is clear the essential cases are the fundamental
$G/P\subset \BP V_{\o}$ where $\o$ appears at the end of a Dynkin diagram.
The primary difficulty in finding equations is locating
vector bundles $E$ such that $E$ and $E^*\ot L$ both have sections.
If $E$ corresponds to a node on the interior of a Dynkin diagram,
then $E^*$ will have a $-2$ over the   node of $\o$ and $E^*\ot L$ a
minus one on the   node, and thus no sections.
However if $E$ corresponds to a different end of the diagram, one obtains
nontrivial equations.

\begin{example}$E_n/P_{\a_1}\subset \BP V_{\o_1}$ for $n=6,7,8$.
Taking $E=E_{\o_n}$, then $E^*\ot L=E$ and 
$H^0(E_{\o_n})=V_{\o_n}$.
The rank of $E$ is   $2(n-1)$ and the fiber is the
standard representation of $SO(2n-2)$.
$\tdim V_{\o_n}$ is respectively $27,56,248$, so one gets equations
for $\s_r(E_n/P_{\a_1})$ for $r$ respectively up to $2,4,17$.
Since the matrix is square, it is possible these equations have
lower degrees than one naively expects. Especially since in the case
of $E_6$, the secant variety is known to be a cubic hypersurface.
\end{example}

\begin{example} $D_n/P_n\subset \BP V_{\o_n}$. Here we may take
$E=E_{\o_1}$ and then $E^*\ot L=E_{\o_{n-1}}$. The fiber of
$E$ is the standard representation of $A_{n-1}$
in particular it is  of dimension $n$, $H^0(E)=V_{\o_1}$
which is of dimension $2n$ and $H^0(E^*\ot L)=V_{\o_{n-1}}$.
Thus these give (high degree) equations for the spinor varieties
$D_n/P_n$ but no equations for their secant varieties.
Some equations for $\s_2(D_n/P_n)$ are known, see
\cite{manispin,LWchss}
\end{example}

\bibliographystyle{amsplain}
\bibliography{giorgio,Lmatrix}
\end{document}